\documentclass{article}

\usepackage{arxiv}

\usepackage[utf8]{inputenc} 
\usepackage[T1]{fontenc}    
\usepackage[hidelinks]{hyperref}       
\usepackage{url}            
\usepackage{booktabs}       
\usepackage{amsfonts}       
\usepackage{nicefrac}       
\usepackage{microtype}      
\usepackage{graphicx}
\usepackage{natbib}
\usepackage{doi}

\usepackage{orcidlink}

\usepackage{bbm}
\usepackage{amsmath}
\usepackage{amsthm}
\usepackage{subcaption}

\newcommand{\calG}{\mathcal{G}}
\newcommand{\calH}{\mathcal{H}}
\newcommand{\calM}{\mathcal{M}}

\newcommand{\calA}{\mathcal{A}}
\newcommand{\calC}{\mathcal{C}}

\newtheorem{proposition}{Proposition}

\title{Underground Freight Transportation for Package Delivery in Urban Environments}

\author{ Sarah Powell \orcidlink{0000-0002-4919-5545}, Ann Melissa Campbell \orcidlink{0000-0002-9927-687X}, Mojtaba Hosseini \orcidlink{0000-0003-0272-016X} \\
	Department of Business Analytics\\
	University of Iowa\\
	Iowa City, IA 52242 \\
	\texttt{sarah-powell@uiowa.edu, ann-campbell@uiowa.edu, mojtaba-hosseini@uiowa.edu} \\
}

\renewcommand{\headeright}{}
\renewcommand{\undertitle}{}
\renewcommand{\shorttitle}{Underground Freight Transportation for Package Delivery}

\hypersetup{
pdftitle={A template for the arxiv style},
pdfsubject={q-bio.NC, q-bio.QM},
pdfauthor={David S.~Hippocampus, Elias D.~Striatum},
pdfkeywords={First keyword, Second keyword, More},
}

\begin{document}
\maketitle

\begin{abstract}
The use of underground freight transportation (UFT) is gaining attention because of its ability to quickly move freight to locations in urban areas while reducing road traffic and the need for delivery drivers. Since packages are transported through the tunnels by electric motors, the use of tunnels is also environmentally friendly. Unlike other UFT projects, we examine the use of tunnels to transport individual orders, motivated by the last mile delivery of goods from e-commerce providers. The use of UFT for last mile delivery requires more complex network planning than for direct lines that have previously been considered for networks connecting large cities. We introduce a new network design problem based on this delivery model and transform the problem into a fixed charge multicommodity flow problem with additional constraints. We show that this problem, the $n_d$-UFT, is NP-hard, and provide an exact solution method for solving large-scale instances. Our solution approach exploits the combinatorial sub-structures of the problem in a cutting planes fashion, significantly reducing the time to find optimal solutions on most instances compared to a MIP. We provide computational results for real urban environments to build a set of insights into the structure of such networks and evaluate the benefits of such systems. We see that a budget of only 45 miles of tunnel can remove 42\% of packages off the roads in Chicago and 32\% in New York City. We estimate the fixed and operational costs for implementing UFT systems and break them down into a per package cost. Our estimates indicate over a 40\% savings from using a UFT over traditional delivery models. This indicates that UFT systems for last mile delivery are a promising area for future research.
\end{abstract}


\section{Introduction} \label{intro}

Underground freight transportation (UFT), sometimes called underground logistics systems, is an unmanned transportation system in which crates of packages travel underground through tubes between different locations. The use of UFT is gaining attention because of its ability to quickly move freight to locations in urban areas while reducing road traffic and the need for drivers, which have been increasingly in short supply. Since packages are transported through the tunnels by electric motors, the use of tunnels is environmentally friendly, and the operation of the tunnels is not hampered by weather conditions such as snow storms.  Most of the proposed applications for using tunnels for freight transportation focus on moving containerized cargo around ports such as with HyperloopTT \citep{HyperloopTT}, moving palletized goods between large urban areas such as with Cargo Sous Terrain in Switzerland \citep{cargo}, or moving palletized goods into a city center via the Smart City Loop in Hamburg \citep{Smart}. These proposals require very large capital investments and several years for construction. 
To the best of our knowledge, only the Cargo Sous Terrain is under construction with the goal of connecting key hubs in Switzerland by 2031.
Unlike these projects, our paper examines the use of tunnels to transport individual orders, motivated by the last mile delivery of goods from e-commerce providers such as Amazon or delivery service providers such as UPS or FedEx.  

The use of tunnels for last mile delivery has been considered by \cite{boysen2023optimization} and \cite{boysen2024jam} but from a scheduling and operational perspective for a given tunnel network. To the best of our knowledge, the current literature does not address the question of how to design a UFT network for individual package delivery or evaluate the potential benefits of such a system. 
The use of UFT for last mile delivery in urban environments requires more complex network planning than for direct lines between cities such as for the Cargo Sous Terrain.  There are strategic decisions about what parts of a city should be served by tunnels and which should remain served by trucks, as well as constraints on the design of the network.

This paper is inspired by meetings with the London-based company Magway. Magway \citep{magway} proposes the use of 90 cm wide tunnels built alongside an existing road or transit network. These tunnels are significantly smaller than those proposed for use by Cargo Sous Terrain and the Smart City Loop, resulting in a lower capital expense for construction while increasing the feasibility of using existing space along road or transit networks or re-using existing infrastructure, such as tunnels built for mail distribution in large cities. Magway has developed the technology for a UFT system utilizing electric motors to propel carriages containing individual orders along tracks in tunnels. 
The linear motors designed by Magway require very little service and maintenance. Magway proposes installing two parallel tunnels for inward and outward moving goods. Part of the network design is that the two tunnels interconnect every kilometer, so goods can be redirected between tunnels in the case that tunnel maintenance is needed.  In this situation, goods would move in only one direction at a time while a section of tunnel is repaired or replaced.
With a capacity of 72,000 UFT carriages per hour, the number of trucks used for package delivery could be greatly reduced by installing such a system. A UFT system, such as the one proposed by Magway, can be built and used by a single delivery carrier, such as UPS, or built by a city or government entity that would charge multiple service providers for the use of time and capacity.  We consider the latter here.

In this paper, we consider a city that has decided to build a UFT system for package delivery. The system will deliver packages from a depot(s) to microhubs that each serve a small area or neighborhood. Customers can walk or bike to retrieve their packages from such a center, or packages can be delivered ``the last meter" by a service provider using a cargo bike or small delivery vehicle. The use of microhubs is becoming increasingly popular to help reduce the number of vehicles in urban areas.  
Interestingly, there is little agreement on the correct form of a microhub.  Some cities lease parking spots to serve as microhubs (e.g. Washington, DC; London) \citep{NYC}, while others use shipping containers on the side of the road (e.g. Berlin) \citep{NYC}, on- or off-street locations with truck parking and security (e.g. New York City) \citep{micro}, pods in parks (e.g. New York City) \citep{micro}, or reserved areas in public garages (e.g. New York City, London) \citep{micro, NYC}. Some cities are utilizing unused commercial space as microhubs (e.g. Amsterdam, Paris) \citep{south, NYC}. Most microhubs are in the range of size of 500-1000 square feet, as they only serve a small area, but some can be 1000 to 10,000 square feet \citep{Locus}. Although microhubs can take many forms, they all have the focus of allowing goods to be transferred to cleaner modes of final transport to the end customer \citep{lee}.

In our model, the UFT system will be combined with traditional trucks and vans to cover a large service area, such as a large city. For example, the UFT network may cover high demand areas, but delivery trucks will be used to serve the rest of the network.  We present a model to decide where to build UFT links such that the amount of demand served by UFT is maximized while a tunnel budget and maximum loading capacity are respected. Cities considering the installation of the UFT systems, to our understanding, are most interested in reducing the number of trucks on the road network (Magway Interview \citep{magway}), so our objective of maximizing the demand served allows cities to gain the most benefit from the UFT network. We focus on this problem at a strategic level to understand the potential benefits of such a delivery model, such as the number of deliveries that could be served, trucks taken off the road, and savings in emissions. We discuss the advantages of using UFT as compared with a traditional delivery system using electric vehicles. We also provide an estimate of the building and operating costs to understand the financial implications of such as system.  

We note that our model assumes some collaboration among package delivery service providers. In most of the existing models for collaboration for freight delivery, deliveries from different providers are combined on vehicles.  The required knowledge sharing regarding customers and cost allocation are challenging parts of the implementation {\citep{gansterer2020shared}}. As in the approach for tunnel use presented in \cite{boysen2024jam}, we assume service providers will book and pay for time and capacity, allowing collaboration without the same challenges. We assume that the ``last meter" delivery from the microhubs could either be divided amongst the original service providers or done by a third party.  Similarly, deliveries to parts of the city not served by the UFT network may be handled by the original service providers or in a collaborative manner.  In some of our experiments, we compare the results of using a UFT system with a fully collaborative truck-based system, as this represents a lower bound on delivery costs.

Our major contributions include introducing a new problem based on a promising new delivery model and transforming the problem into a fixed charge multicommodity flow problem with additional constraints. A special case of this problem with a single depot and unit distances between microhubs is the $k$-cardinality tree problem, which is known to be NP-hard \citep{fischetti1994weighted}. By decomposing the model into a base model with cuts added as needed, we offer an improved methodology capable of quickly solving many realistically sized problem instances to optimality.  We provide computational results for real urban environments to build insight into the structure of such networks and evaluate the benefit of such systems. Specifically, we use population and road network data for the cities of Chicago and New York City in our computational experiments. 
We will show that a tunnel of only 45 miles in length can take 42\% of packages off the roads in Chicago and 32\% in New York City. At the same tunnel budget of 45 miles, serving microhubs with a UFT system can save 5,267 truck miles per day in Chicago and 5,660 truck miles per day in New York City compared to a collaborative delivery setting. We also estimate the fixed and operational costs for implementing UFT systems and break them down into a per package cost. Our estimates indicate over a 40$\%$ savings from using a UFT over traditional delivery models.  This indicates that UFT systems for last mile delivery are a promising area for further research.

\section{Literature review}
We review related literature on underground freight transportation, the use of tunnels for high-speed passenger transportation, collaborative transportation, and multicommodity network flow problems. And, since the depot location and allocation of microhubs to depots is a decision in our model, we note that \cite{alumur2021perspectives} provide a taxonomy of hub location problems.

\subsection{Underground freight transportation}
Underground tunnels or tubes have been used as a  delivery method for years. As described in \cite{visser2018development}, tunnels using air propulsion have historically been used to deliver telegrams and mail between or within buildings. And, underground connections in Central London have been used to move mail between post offices \citep{visser2018development}. 

Many of the more recent studies of UFT systems focus on using them at ports and/or with pallets of goods, rather than individual packages. For example, the impact of using a UFT network to move palletized goods into and out of a port in Shanghai is analyzed in \cite{chen2018using}. The authors show that  such a network would provide benefits in terms of transportation costs, time, and truck emissions. The financial benefits and feasibility of building a UFT system to move containers of goods from the Port of Houston to the city of Dallas is provided in \cite{ehsan2017investment}.  In \cite{liu2004feasibility}, the authors test the feasibility of using pneumatic capsule pipelines for transporting pallets of goods in New York City. They found that a tube system could deliver most of the cargo currently carried by trucks, including those carried by tractor trailers. But, for this delivery system to be effective, there must be an extensive network of large underground tunnels built throughout the city with numerous inlet/outlet stations. They summarize that implementing such a system would be costly, but the benefits from reduced truck traffic in the city would be huge. We note that the construction costs of the larger tunnels proposed by the authors are much more expensive than the smaller ones discussed in this paper. 

In other evaluations of UFT, the authors focus on simulating the performance of a system rather than its design.  For example, \cite{dong2019impact} use a systems dynamic method to analyze the quantitative relationship between the implementation strategy of UFT and the sustainability of urban transportation and logistics. They use a simulation study with data from Beijing to model the benefits to urban traffic and logistics by using a UFT system to supplement an existing urban transport system. They consider various UFT network densities which allows them to determine the volume of packages that can be delivered via UFT and the impact this will have on congestion, delivery time, and emissions. 

In an urban setting similar to ours, \cite{boysen2023optimization} and \cite{boysen2024jam} consider the UFT problem from an operational perspective, addressing the coordination of packages from a depot outside a city to inner city microhubs using UFT tunnels. They do not consider the network design. In \cite{boysen2023optimization} they consider limitations on the storage capacity of microhubs and the transport capacity of a delivery person at a microhub when determining the optimal way to assign shipments to arrival slots. The authors show that emission reductions come at the price of excessive bike traffic in urban areas when a microhub serves a large area and e-bikes are used for delivery to customers. In \cite{boysen2024jam} they evaluate the tunnel scheduling problem to determine when the capacity of tunnels is significantly under utilized. They found that when many logistics providers, each with a small transport volume, book tunnel capacity by time slot, there is a significant amount of unused tunnel capacity. It is also the case that urgent same-day deliveries result in significant under utilization of cargo tunnels, and that small and time critical transport volumes booked by time slot can result in increased road traffic when trucks must make deliveries than cannot be completed on time by tunnel. We consider a similar setting where packages from multiple service providers will enter the UFT network at a depot located outside the service area and will be sent via tunnels to microhubs located within the city. But, we assume that the tightest delivery deadlines are next day and would not result in delivery deadline or microhub capacity infeasibilities.

\subsection{Tunnels for high-speed passenger transportation}
While using tunnels for freight transportation has been considered for years, the idea of high speed passenger transportation using sealed tubes, such as \cite{Hyperloop}, has created new interest. The goal of Hyperloop One is to provide highspeed travel between major cities, such as a 30 minute connection between Helsinki and Stockholm instead a 3.5 hour flight or overnight ferry trip. While such high speed transportation can be a significant time saver for commuters and a generator of economic benefits for cities and regions along the route, these networks are often proposed as single connections between high demand urban areas, unlike the multi-station network required for a UFT system serving many microhubs. 

There are a few papers, such as \cite{merchant2020towards} and \cite{shah2019hyperloop}, that consider hyperloop networks with multiple stops, but they do not provide exact algorithms to determine the optimal network with respect to maximizing demand served. \cite{merchant2020towards} select cities for hyperloop stations based on given criteria. In \cite{shah2019hyperloop}, a heuristic for building a network is proposed based on the benefit to cost ratio of links. The benefit of a link is determined by the travel time savings. While this provides insight for designing a passenger transport system, it is less appropriate for designing a freight transportation system that seeks to maximize the number of packages delivered via UFT.

\subsection{Collaboration for package delivery} \label{collaboration}
An underground freight transportation system built and managed by a city will be a shared resource, similar to a road network, and will require some level of collaboration by the different service providers using it. We focus our review of collaborative delivery literature on those that consider a centralized collaboration system with shared depots. We note that \cite{gansterer2020shared} and \cite{gansterer2018collaborative} provide detailed reviews of the work in collaborative logistics.

It is common in collaborative logistics literature to assume that different service providers have access to a common set of depots where packages can be picked up or dropped off. For example, \cite{zhang2020composite} consider depots with orders for multiple carriers and allow carriers to make deliveries to shared depots. They use this assumption to model a multi-objective collaborative vehicle routing problem and evaluate the savings from collaboration. \cite{ko2020collaboration} propose operating one service center in each service area that various companies will share. They derive a last-mile delivery time function to maximize the profit of each participating company. \cite{handoko2014auction} focus on maximizing the profit at an urban consolidation center (UCC) that is located outside of a city center. Various service providers use the system and pay to use the center. They assume the center owns a fleet of vehicles for deliveries and propose a profit-maximizing auction mechanism for the use of the UCC. This is similar to the system used at La Petite Reine in Paris,  Westfield Consolidation Centre in London, and Binnenstadservice.nl in Nijmegen, the Netherlands. Similarly, we assume that service providers will pay to use the UFT system to deliver packages to microhubs located throughout the city. We also assume that delivery service providers using UFT to deliver packages to microhubs will first deliver packages to a depot located outside of the service area. 

\subsection{Multicommodity flow problems}
The formulation we use for UFT network design is based on adapting a formulation for the integer multicommodity flow problem, which is known to be NP-hard \citep{even1975complexity}.  The multicommodity flow problem is  well-studied in the operations research literature and is commonly used to model network design for transportation and delivery problems. It was first introduced in \cite{ford1958suggested} and \cite{hu1963multi}. We note that \cite{salimifard2022multicommodity}, \cite{ouorou2000survey}, and \cite{assad1978multicommodity} provide surveys of applications and solution techniques for multicommodity flow problems.

Multicommodity flow problems have many applications across transportation, logistics, and communication problems. In \cite{yaghini2012multicommodity}, the multicommodity network design problem is applied to rail freight transportation planning. They show that branch-and-bound and decomposition methods are efficient ways to find exact solutions. In another work on modeling freight transport, \cite{rudi2016freight} use a capacitated multicommodity flow formulation to provide decision support in intermodal freight transportation planning. Multicommmodity flow formulations for freight delivery can be extended to the time-constrained case as shown in \cite{hellsten2021transit}. Additional works applying multicommodity flow to service network design for freight transportation are reviewed in \cite{crainic2000service}. They focus on consolidated transportation for long-haul freight transport and provide common solution methods including Lagrangian relaxation, branch-and-bound methods, applying valid inequalities, and cutting plane methods as well as heuristics. \cite{wang2018multicommodity} provides a survey of additional applications and formulations of the multicommodity flow problem. Most of the formulations mentioned here focus on cost minimization when a certain level of demand for each commodity must be satisfied. Instead, we focus on maximizing demand served with a budget constraint on the size of the network. 

Even for very simple multicommodity network flow problems with continuous flow and linear costs, it is practically very hard to find optimal solutions \citep{salimifard2022multicommodity}. And, for problems like network design for UFT with integer flow and fixed costs, finding optimal solutions is even more challenging. \cite{ouorou2000survey} provide a survey of algorithms for multicommodity flow problems and state that the structure of such problems makes decomposition methods attractive.
Along this line, we propose an exact solution method for solving realistically sized instances of UFT to optimality. Inspired by the optimality conditions that we identify for UFT, we leverage the combinatorial sub-structures to simplify the formulation and produce the necessary constraints in a lazy constraint fashion akin to subtour elimination techniques for the travelling salesman problem \citep{cook2011traveling}.

\section{Problem description and model} 

 In this section, we describe the problem in Section \ref{Description}, the mixed integer programming model in Section \ref{Model1}, and properties of the problem in Section \ref{props}.

\subsection{Problem description}\label{Description}

In the $n_d$-UFT, we determine the the design of the network including the links and set of microhubs that should be built to maximize the amount of demand served by UFT. The UFT network has $n_d$ depots, a maximum daily depot loading capacity of $P_{max}$, and a distance budget of $D$. When a UFT link is built, tunnels are installed in both directions at the same time. Thus, once a carriage arrives at a microhub, it can return to the depot along the same set of links by which it arrived. So, unlike a typical truck route, UFT carriages do not need to follow a circuit to return to the depot.

We assume that there is a set of potential microhubs $\mathcal{M}$, with each $i\in \mathcal{M}$ covering a pre-specified area within the city. The demand at a microhub $i \in \mathcal{M}$, $P_i$, represents the average number of packages demanded per day for the particular area served by that microhub. We assume that the demand associated with microhubs that are not served by the UFT will be served by delivery trucks and those microhubs will not be built.

The network design decision is based on selecting a set of arcs, denoted $\mathcal{A}$, between potential microhubs. These arcs represent candidates for where the UFT links may be built. 
The distance budget limits the total length of (two-way) tunnels that can be built along the road network. We denote the distance budget for the UFT network as $D$ and the length of a given arc $(i,j)$ as $d_{ij}$.  The carriages, which contain the packages for a particular destination, can be programmed to follow a specific path to a destination, allowing for multiple arcs adjacent to the same microhub to be selected in potential solutions. This is an important differentiation from prior work since most proposed networks with palletized deliveries consist of only a single line.  

The tunnel network originates at one or more depots. The depots are chosen from a set of potential depot locations $\mathcal{H}$ with the number of selected depots specified as $n_d$. We assume that a city would likely build a low number of very large depots to derive economies of scale from building such facilities, so we will experiment with different low values of $n_d$ in Section \ref{results}.  At the depots, the carriages enter the network from loading/unloading equipment located in a set of bays, as illustrated in Figure \ref{fig:TRL} from a report by \cite{TRL} in collaboration with \cite{magway}.  Loading bays can be dedicated to different providers, for example.  

Due to the high speed of tracks in a UFT system, large numbers of packages can move quickly from depot(s) to destination microhubs.  The capacity of the UFT system is determined by the headway, the time between consecutive UFT carriages passing a set point. This value, $\tau$, is derived based on the speed and length of UFT carriages. To achieve delivery to customers within a service level, we model this via a limitation on the total time, $T$, for packages to enter the network.  For a given $T$, we can derive the capacity on the number of packages, $P_{max}$, by $T/\tau$.  Thus, if we set $P_{max}$ as the loading capacity for a depot, all the packages assigned to a depot can enter the network in time $T$ or less. For example, a city may want the capability of all packages entering the network within 12 hours ($T$) to enable next day delivery, and thus $P_{max}$ would be the number of packages that can be loaded in 12 hours.  We note that the model can easily be adapted to run multiple times per day by modifying $T$ and $P_i$ values appropriately (see Proposition \ref{demand} for more details). Given that the capacity is defined by the size of carriages and how quickly they can enter the tunnel network from a depot, we assume that each depot serves a separate network, and we will assume each depot will only have one arc entering and exiting from it.

We make a few assumptions that keep the results easier to interpret. For example, we focus on where to build arcs between microhubs rather than the cost of the microhubs.
This approach makes it easier to interpret the changes that result from increasing or decreasing the budget.   

\begin{figure}[!ht]
    \centering
    \includegraphics[width = 6cm]{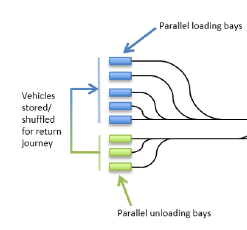}
    \caption{A Schematic Diagram for Depot Loading Design}
    \label{fig:TRL}
\end{figure}

\subsection{Model formulation} \label{Model1}
We formulate the $n_d$-UFT using a multicommodity flow formulation similar to that presented in \cite{kandyba2011exact} but add several modifications to account for the differences in our problem. \cite{kandyba2011exact} transform the undirected arcs of their multicommodity network flow problem into directed arcs by replacing each undirected arc with two directed arcs. This better captures the flow out of a specified root node in their problem. Similarly, we transform the undirected arcs of the $n_d$-UFT into directed arcs to specify the path carriages would travel between a depot and microhub, but note that carriages may still travel both directions along any connection in the network. We denote the set of directed arcs as $\mathcal{A'}$.

The mixed integer programming (MIP) formulation is given in \eqref{obj}--\eqref{mcf8}, and the set of parameters is given in Table \ref{Tab:Parameters}. The binary decision variables $y_i$ for $i \in \mathcal{M} \cup \mathcal{H}$ indicate whether node $i$ is included in the UFT network, and the binary decision variables $x_{ij}$ for $(i,j) \in \mathcal{A'}$ indicate if a UFT tunnel is built along the arc $(i,j)$. The flow variables $f^{v}_{ij}$ denote the flow of commodity $v$ along the arc $(i,j)$. Each node $v \in \mathcal{M}$ is assigned its own commodity. Definitions of these variables are given in Table \ref{Tab:Variables1}.

\begin{table}[h]
\begin{tabular}{ p{20mm} p{125mm} } 
\hline
  Parameter & Description \\
\hline
  $\calM$ & Set of microhubs \\
  $\calH$ & Set of candidate depot locations\\
  $\mathcal{A'}$ & Set of directed arcs between nodes in {$\calM \cup \calH$} \\
  $n_d$ & Number of depots \\
  $P_i$ & Number of packages to be delivered to microhub $i \in \calM$ per day \\
  $P_{max}$ & Number of packages that can be loaded in $T$ hours \\
  $T$ & Total permitted loading time per day \\
  $D$ & Maximum length of the UFT network \\
  $d_{ij}$ & Distance between nodes $i$ and {$j \in \calM \cup \calH $}\\
  \hline 
\end{tabular}
  \caption{Definition of parameters.}
  \label{Tab:Parameters}
  \end{table}

\begin{table}[h]
\begin{tabular}{ p{20mm} p{125mm}  } 
\hline
   Variable & Description \\
\hline
   $y_{i}$ & $y_{i} = 1$ if the node $i \in \calM \cup \calH$ is included in the UFT network \\
   $x_{ij}$ & $x_{ij} = 1$ if the directed arc $(i,j) \in \mathcal{A'}$ is chosen for building a UFT link \\
   $f^v_{ij}$ & The flow of commodity $v \in \calM$ along the directed arc $(i,j) \in \mathcal{A'}$\\
\hline	
\end{tabular}
  \caption{Definition of variables.}
  \label{Tab:Variables1}
  \end{table}

\begin{align}
    \max \quad & \sum_{i \in \calM} P_{i} y_{i} \label{obj} \\
    \text{s.t.} \quad & \sum_{j:(j,i) \in \mathcal{A'}} f_{ji}^{v} - \sum_{j:(i,j) \in \mathcal{A'}} f_{ij}^{v} = y_v && \forall i \in \calM, \text{ } \forall v \in \calM \text{ if } i = v  \label{mcf4}\\
    & \sum_{j:(j,i) \in \mathcal{A'}} f_{ji}^{v} - \sum_{j:(i,j) \in \mathcal{A'}} f_{ij}^{v} = 0 && \forall i \in \calM, \text{ } \forall v \in \calM \text{ if } i \neq v  \label{mcf5}\\
    & \sum_{(i,j) \in \mathcal{A'}} d_{ij} x_{ij} \leq D \label{c1}\\
    & \sum_{j: (h,j) \in \mathcal{A'}} x_{hj} = y_h && \forall h \in \calH \label{c2}\\
    & \sum_{h \in \calH} y_h = n_d \label{mcf10}\\
    & \sum_{i:(i,j) \in \mathcal{A'}} x_{ij} = y_j && \forall j \in \calM \label{c4}\\
    & \sum_{v \in \calM} \sum_{j:(h,j) \in \mathcal{A'}} P_v f^v_{hj}  \leq P_{max} && \forall h \in \calH \label{c3}\\
    & 0 \leq f_{ij}^{v} \leq x_{ij} && \forall (i,j) \in \mathcal{A'}, \forall v \in \calM \label{mcf6}\\
    & y_i \in \lbrace 0,1 \rbrace && \forall i \in \calM \cup \calH \label{mcf7}\\
    & x_{ij} \in \lbrace 0,1 \rbrace && \forall (i,j) \in \mathcal{A'} \label{mcf8}
\end{align}

The objective of the $n_d$-UFT problem given in \eqref{obj} is to maximize the number of packages delivered by the UFT network.  Constraints \eqref{mcf4} and \eqref{mcf5} conserve the flow of commodities through the network. Constraint \eqref{c1} ensures that the total length of tunnels built does not exceed $D$. Due to the infrastructure associated with building a depot, we assume there will be exactly one arc from each depot to the set of microhubs. This is reflected in Constraint \eqref{c2}. Constraint \eqref{mcf10} in combination with constraint \eqref{c2} determines the number of depots that are included in the solution. Constraint \eqref{c4} ensures that each depot serves a separate network by allowing at most one directed arc to be incident to each microhub. Constraint \eqref{c3} sets the depot loading capacity of each depot to be the number of carriages that can be loaded within the time limit $T$. Constraint \eqref{mcf6} requires the flow along an arc to be non-negative. Constraints \eqref{mcf7} and \eqref{mcf8} define the binary variables.

\subsection{Properties of \texorpdfstring{$n_d$}-UFT}\label{props}
\begin{proposition}
The $n_d$-UFT problem is NP-hard.
\end{proposition}
\begin{proof} 
We show that this problem is NP-hard by transforming the rooted $k$-cardinality tree problem (KCT), which is known to be NP-hard \citep{fischetti1994weighted}, to our problem. The rooted KCT problem consists of finding a minimum node-weighted rooted tree with exactly $k$ arcs. When all arcs in a UFT network are of unit length, then the constraint requiring exactly $k$ arcs to be included in the network is equivalent to the following constraint on the total length of arcs included: 
\begin{align*}
    \sum_{(i,j) \in \mathcal{A'}} d_{ij}x_{ij} \leq k.
\end{align*} 
\noindent
 While the KCT problem as presented in \cite{fischetti1994weighted} seeks to minimize the sum of the weights, the $n_d$-UFT maximizes the demand served through the UFT network based on the demand at microhubs or nodes in the network. So, the node weights from the KCT problem will have weights flipped to create negative $P_i$ values.
 If we include a non-binding loading time constraint (e.g. greater than the sum of weights) and set the specified root of the rooted $k$-Cardinality tree problem to be the only item in the set $\calH$, then the solution for the $n_d$-UFT problem with $n_d = 1$ will yield the solution for the rooted $k$-cardinality tree problem.
\end{proof}

\begin{proposition}
    The total demand served by the UFT system will be equivalent for a total loading time limit $T$ regardless of the number of equal-length loading windows when the demand for each microhub is divided evenly across loading windows. \label{demand} 
\end{proposition}

\begin{proof} For a loading time limit $T$ for a single loading window, the constraint on the number of packages that can be served from each depot is given by:
\begin{align*}
    &\sum_{v \in \calM} \sum_{j:(h,j) \in \mathcal{A'}} P_{v}f^{v}_{hj} \leq P_{max} = T/\tau && \forall h \in \calH
\end{align*}

\noindent When multiple loading windows are used, and $T$ is divided into $w$ equally sized loading windows of length $t$, and the demand at each microhub is evenly divided among the $w$ loading windows, then we will have an equivalent MIP for each loading window $w$. The constraint on the number of packages that can be delivered from each depot in each loading window is given by:
\begin{align*}
    &\frac{1}{w}\sum_{v \in \calM} \sum_{j:(h,j) \in \mathcal{A'}} P_{v}f^{v}_{hj} \leq \frac{T}{w \tau} && \forall h \in \calH
\end{align*}
This is equivalent to the constraint on the maximum number of packages that can be delivered in the single loading window case, and the remaining constraints of the problem remain the same. The value of each microhub $i \in \calM$ for each loading window is $P_{i}/w$, giving an objective function of:
\begin{align*}
    \max \sum_{i \in \calM} \frac{P_i}{w}y_i.
\end{align*}
For all of the loading windows, since the demand at each microhub is divided evenly, we will obtain equivalent solutions. The total demand served by the system can be obtained by multiplying the objective value of a specific loading window by $w$ or by writing the objective to the MIP for a single loading window as:
\begin{align*}
    \max\; w\sum_{i \in \calM} \frac{P_i}{w}y_i = \max \sum_{i \in \calM} P_i y_i
\end{align*}
This makes this problem equivalent to the version with a single loading window of length $T$. Thus, for any number of equally sized loading windows that sum to $T$ with demand at each microhub evenly spread across these windows, the optimal UFT network will be the same as the case of a single loading window with length $T$.
\end{proof}

\begin{proposition}
An optimal solution of the $n_d$-UFT will be a forest.\label{prop:forest}
\end{proposition}
\begin{proof}
To prove that the created $n_d$ solution will be a forest, we consider the contradiction.  If there is a cycle (directed or undirected) in a graph, select any arc that is part of that cycle.  Removing the arc will still allow the same set of nodes to be visited by the tunnel network and at a lower tunnel building cost.  It does not matter which arc in the cycle is selected because tunnels will be built in both directions along the arcs in the final solution.
\end{proof}

\section{Solution Methodology}

{We introduce a solution method for solving large-scale instances of the $n_d$-UFT. Our solution approach relies on exploiting the combinatorial sub-structures of the problem in a cutting planes fashion.} {We describe our model decomposition and reformulation in Section \ref{sec:decomp}, feasibility cuts and valid inequalities in Section \ref{sec:mp-feasibility}, and the implementation of our cutting planes method in Section \ref{sec:implementation}.}

\subsection{Model Decomposition} \label{sec:decomp}
There are two types of variables in the MIP formulation of the $n_d$-UFT from Section \ref{Model1}. The binary variables $x_{ij}$ and $y_i$ indicate which depots, microhubs, and UFT links should be included in the optimal solution. The continuous flow variables $f_{ij}^v$ 
ensure that the binary variables $x_{ij}$ and $y_i$ satisfy the connectivity (constraints \eqref{mcf4}--\eqref{mcf5}) and capacity requirements (constraint \eqref{c3}) of the $n_d$-UFT. We can decompose the MIP into a master problem containing only the binary variables $x_{ij}$ and $y_i$.   The master program (MP) is given by \eqref{c1}-\eqref{c4} and \eqref{eq:mp-obj}-\eqref{eq:mp-8}. 

\begin{align}
\max \quad & \sum_{i \in \calM} P_i y_i \label{eq:mp-obj}\\
    \text{s.t.} \quad & \eqref{c1} - \eqref{c4} \nonumber \\
    & x_{ij} \leq y_i && \forall (i,j) \in \mathcal{A'}: i \notin \calH \label{eq:mp-6}\\
    &\sum_{(i,j) \in \mathcal{A'}} x_{ij} = \sum_{j \in \calM \cup \calH} y_j - n_d \label{eq:mp-forest-nodes-arcs} \\
    & y_i \in \lbrace 0,1 \rbrace && \forall i \in \calM \cup \calH \label{eq:mp-7}\\
    & x_{ij} \in \lbrace 0,1 \rbrace && \forall (i,j) \in \calA\label{eq:mp-8}
\end{align}

As in the MIP, the objective of MP \eqref{eq:mp-obj} and constraints \eqref{c1} $-$ \eqref{c4} maximize the amount of demand served by the UFT network while ensuring that the budget on the total length of the UFT network is respected, there is at most one arc out of each depot, $n_d$ depots are included in the solution, and at most one directed arc is incident to each microhub. From Proposition \ref{prop:forest}, we know that the solution to the $n_d$-UFT will be a forest of $n_d$ trees, so we can add constraint \eqref{eq:mp-forest-nodes-arcs} which limits the number of arcs in the solution to be the number of nodes reached minus the number of trees. Constraints \eqref{eq:mp-7} and \eqref{eq:mp-8} define the domain of the binary variables $y_i$ and $x_{ij}$.

Without the flow variables, the MP has many fewer variables than the MIP, but may return solutions that violate the connectivity and/or capacity requirements of the MIP.  We can iteratively add new feasibility constraints that ensure that $x_{ij}$ and $y_i$ meet the connectivity and capacity requirements. These new constraints using binary variables are presented in Section \ref{sec:mp-feasibility}.  
\subsection{A Combinatorial Characterization of the Feasibility Cuts and Valid Inequalities }\label{sec:mp-feasibility}
We first describe the capacity and connectivity constraints followed by additional valid inequalities (VI).

\subsubsection{Feasibility cuts}\label{feas}

\paragraph{Capacity Constraints:} {To enforce the capacity constraints \eqref{c3}, we must ensure that any connected component of the graph {induced by a solution to the MP which exceeds the depot capacity is avoided.}
For the $n_d=1$ case, it suffices to enforce the constraint $\sum_{i\in \calM} P_iy_i\le P_{max}$. 
However, for $n_d\ge 2$, for any subset of microhubs $\hat{\calM}\subseteq \calM$, we can allow these nodes to be part of a connected component
only if $\sum_{i\in \hat{\calM}} P_iy_i \le P_{max}$. {We can therefore model the capacity constraints using constraints of type \eqref{eq:cut-capacity} on each set $\hat{\calM}$ where the sum of the demand at the microhubs in $\hat{\calM}$ exceeds the depot capacity.}

\begin{align}
    \sum_{(i,j) \in \calA'\cap \hat{\calM}\times \hat{\calM}} x_{ij} \leq |\hat{\calM}|-2 \qquad \forall \hat{\calM}\subseteq \calM: \sum_{i\in \hat{\calM}} P_i > P_{max}\label{eq:cut-capacity}
\end{align}}

\paragraph{Connectivity Constraints:} 
Proposition~\ref{prop:forest} dictates that the solution will have exactly $n_d$ connected components, each connected to one open depot. To enforce connectivity, any subset of microhubs $\hat{\calM}\subseteq \calM$ that includes an open microhub must have an arc entering $\hat{\calM}$ from the set of the remaining nodes in the graph. Therefore, we can enforce connectivity using the set of constraints defined in \eqref{eq:cut-connectivity}.
\begin{align}
    \sum_{i\in \calH \cup \calM\setminus \hat{\calM}} \sum_{j\in \hat{\calM}: (i,j)\in \calA} x_{ij} \ge y_k \qquad \forall k\in \hat{\calM}, \forall \hat{\calM}\subseteq \calM \label{eq:cut-connectivity}
\end{align}

\subsubsection{Valid inequalities}
Adding all of the constraints of type \eqref{eq:cut-capacity} and \eqref{eq:cut-connectivity} will result in connected solutions that meet the depot capacity. However, we can reduce the number of cuts of type \eqref{eq:cut-capacity} and \eqref{eq:cut-connectivity} that need to be added by including new valid inequalities that strengthen our formulation. 

\paragraph{Depot Capacity VI:} This valid inequality  prevents sets of microhubs whose total demand exceeds the depot capacity from being included in the same connected component. We avoid these solutions by adding cuts defined in \eqref{capacityCut} to the MP, and we show that this cut is valid in Proposition \ref{Prop:inequality}.
\begin{align}
    \sum_{(i,j) \in \calA': i \in \hat{\calC}, j \in \hat{\calC}} x_{ij}P_j \leq (1-\sum_{(i,j) \in A': i \notin \hat{\calC}, j \in \hat{\calC}}x_{ij})P_{max} + n_dP_{max} \sum_{(i,j) \in \calA': i \notin \hat{\calC}, j \in \hat{\calC}}x_{ij} \label{capacityCut}
\end{align}

\begin{proposition} \label{Prop:inequality}
    Consider a set of nodes $\hat{\calC}$ that contains microhubs $\hat{\calM}$ and a depot $\hat{h}$, such that $\hat{\calM}$ exceeds the depot's capacity. Then \eqref{capacityCut} is valid for the $n_d$-UFT.
\end{proposition}
\begin{proof}
Inequality \eqref{capacityCut} requires the sum of the UFT satisfied demand from $\hat{\calM}$ to be less than $P_{max}$ if there are no arcs from $(\calM \cup \calH) \setminus \hat{\calC}$ into the set of microhubs and depot in $\hat{\calC}$ (i.e. all the nodes in $\hat{\calC}$ included in a MP solution are in the same connected component). When this is the case, $\sum_{(i,j) \in \calA': i \notin \hat{\calC}, j \in \hat{\calC}} x_{ij}$ will be equal to zero, and the right hand side of \eqref{capacityCut} will be $P_{max}$.
But, if there are arcs going into the set $\hat{\mathcal{C}}$, then cut \eqref{capacityCut} should not be binding since the microhubs in $\hat{\mathcal{M}}$ may be part of different connected components. In this case, the total demand associated with the nodes in $\hat{\mathcal{M}}$ can be as large as $n_d P_{max}$ in feasible solutions. When this occurs and there are $z$ edges from the set $(\mathcal{M} \cup \mathcal{H}) \setminus \hat{\mathcal{C}}$ in to the set $\hat{\mathcal{C}}$, the right hand side of \eqref{capacityCut} becomes Equation \ref{ref2}, making \eqref{capacityCut} non-binding.
\end{proof}
\begin{align}\label{ref2}
    (1-z)P_{max} + zn_dP_{max} = (1+z(n_d-1))P_{max} \geq n_dP_{max}
\end{align}

\paragraph{Depot Connectivity VI:}
In addition to the connectivity constraints \eqref{eq:cut-connectivity}, we can further enforce connectivity to the depots when an MP solution contains microhubs that are disconnected from the depots. 
Let $\calC_1, \calC_2, \dots, \calC_{n_d}$ be the components containing depots $h_1, h_2,\dots,h_{n_d}$, $\calC=\cup_{k=1}^{n_d} \calC_{k}$, and $k\in \calM\setminus\calC$ be a disconnected microhub. Then constraint \eqref{eq:cut-connectivity-depots} requires an arc out of one of the components containing a depot when depots $h_1, h_2, \dots, h_{n_d}$ and $k$ are included in the solution to the master problem.
\begin{align}
    \sum_{i\in \calC} \sum_{j\in (\calH\cup\calM)\setminus\calC: (i,j)\in \calA} x_{ij} \ge y_{k} + \sum_{k=1}^{n_d} y_{h_k} - n_d  \label{eq:cut-connectivity-depots}
\end{align}

\subsection{Implementation}\label{sec:implementation}
Given the exponentially many capacity and connectivity constraints in Section \ref{feas}, including each of these constraints in the MP is not feasible. However, we can iteratively add violated constraints to the MP in a cutting planes fashion until the solution to the MP becomes feasible.

To avoid trivially infeasible solutions, we explicitly add some constraints to the MP to reduce the number of required cuts. In particular, we add the simple cycle elimination constraints \eqref{eq:mp-9} and upper bounding constraints \eqref{eq:mp-10} to the MP.
\begin{align}
    & x_{ij} + x_{ji} \leq 1 && \forall (i,j) \in \mathcal{A'}: (j,i) \in \mathcal{A'}, i<j \label{eq:mp-9} \\
    & \sum_{i \in \calM} P_i y_i \leq P_{max}n_d \label{eq:mp-10}
\end{align}

Constraint \eqref{eq:mp-10} provides a bound on the total amount of demand that can be served by the UFT system based on the depot loading capacity. 
This is a modification of constraint \eqref{c3} that enforces an upper bound on the total flow through the system since we cannot track the flow associated with each depot without the flow variables.  

\subsubsection{Feasibility cuts.} 

\paragraph{Capacity constraints:} 
For a given MP solution $(\hat{\mathbf{x}}, \hat{\mathbf{y}})$, let $\calG(\hat{\mathbf{x}}, \hat{\mathbf{y}})$ denote the graph induced by $(\hat{\mathbf{x}}, \hat{\mathbf{y}})$. If $\calG(\hat{\mathbf{x}}, \hat{\mathbf{y}})$ contains exactly $n_d$ connected components, then by \eqref{eq:mp-forest-nodes-arcs}, $\calG(\hat{\mathbf{x}}, \hat{\mathbf{y}})$ is a forest and each of its connected components is a tree rooted at a depot. Consequently, we only need to check if the trees satisfy the capacity constraints and add a constraint of the form \eqref{eq:cut-capacity} for each tree (including the set of microhubs $\hat{\calM}$) that exceeds the loading capacity.   There could be constraints added for the entire tree $\hat{\calM}$ and all subtrees of $\hat{\calM}$ that violate the depot capacity.   Instead, to balance the strength of the cuts and the number of constraints added to the formulation, we found it most effective to add a single constraint for a minimally violating subtree $\Bar{\calM} \subseteq \hat{\calM}$. We identify a minimally violating subtree using a quick heuristic that returns a tree with high demand microhubs that are more likely to be included in subsequent MP solutions. We build this minimally violating tree in a greedy fashion.

\paragraph{Connectivity constraints:}
For each disjoint component $\hat{\calC}$ in a solution to the MP, we can add a cut of type \eqref{eq:cut-connectivity} for at least one microhub $k \in \hat{\calC}$. To avoid adding a very large number of weak cuts to the MP, we restrict $k$ to microhubs that are more likely to be included in an optimal solution. In our implementation, we consider the two microhubs in $\hat{\calC}$ that have the largest demand values. 

\subsubsection{Valid inequalities} 

\paragraph{Depot capacity VI:} We add a cut generated by the valid inequality \eqref{capacityCut} for each connected component that violates the depot capacity to strengthen the MP. We add this cut over the full set of microhubs and depot in $\hat{\calM} \cup \{\hat{h}\}$ that are connected and exceed the depot capacity in an MP solution, $(\hat{\mathbf{x}}, \hat{\mathbf{y}})$. This ensures large connected components from a MP solution are split into separate trees in subsequent MP solutions.

\paragraph{Depot Connectivity VI:}
We add a constraint of type \eqref{eq:cut-connectivity-depots} for each disjoint component. Instead of adding many weak cuts of this type, we add a single cut for the microhub $\hat{k}$ in the disjoint component $\hat{\calC}$ that has the highest demand, as this node is more likely to be included in subsequent solutions. In the following section, we compare the performance of our solution method with the MIP presented in Section \ref{Model1}.

\section{Computational results} \label{Results}
Here, we provide a set of computational experiments to show the impact of using a UFT network to serve delivery demand at microhubs located in an urban area, as well as the value of our cutting planes solution method. We provide a description of our experimental design in Section \ref{ExperimentalDesign.Network} and our different types of results in Sections \ref{method} and \ref{results}.

\subsection{Experimental design} \label{ExperimentalDesign.Network}
Using the formulation for the $n_d$-UFT problem in Section \ref{Model1}, we can find optimal solutions for a variety of instances derived using data from Chicago and New York City. 
Both cities have large urban populations but have very different geographical features. For all instances, we use Gurobi software to implement our proposed solution method to solve problem instances to optimality or to the best solution found within a time limit of 5 hours. All solutions found within 5 hours have a gap of 0.01\% or less.

 \subsubsection{Microhubs and depots}
In our test instances, we consider microhubs distributed across both cities that are spaced 1.5 miles apart.  This provides an area with a radius of approximately $3/4$ mile served by each microhub, consistent with estimations of the distance that can be walked in discussions of 15-minute cities \citep{15min}. We use a service area of this size for each microhub so that the last meter of delivery may be completed by customers picking up their own packages at the microhub or deliveries from the microhub could be completed by e-bikes or other environmentally friendly methods. To determine the length of the arcs between adjacent microhubs, we use road network data provided by Open Street Maps (\cite{OpenStreetMap}) via the Veroviz package (\cite{veroviz2020}). 
We assume that depots used for the UFT network will be located outside the area to be covered by UFT, similar to distribution centers or UCCs located outside of a city or urban area.  To show how the use of multiple depots affects the impact of the UFT network, we compute solutions for $n_d = 1$ and $n_d = 2$ for both Chicago and New York City with a candidate set of potential depot locations.

The locations of microhubs for Chicago and New York City are shown in Figure \ref{fig:Demand Centers} with red markers, and potential depot locations are shown with blue markers. An illustration of a solution using direct arcs between microhubs is shown in Figure \ref{fig:ArcExample}, and this same solution mapped to the road network using the shortest road distance between microhubs is shown in Figure \ref{fig:RoadExample}. For clarity, we will present solutions in Section \ref{results} using the method in Figure \ref{fig:ArcExample}, but the length of the arcs in the model is determined by the road network distance between them.

\begin{figure}
     \centering
     \begin{subfigure}[b]{0.313\textwidth}
         \centering
         \includegraphics[width=\textwidth]{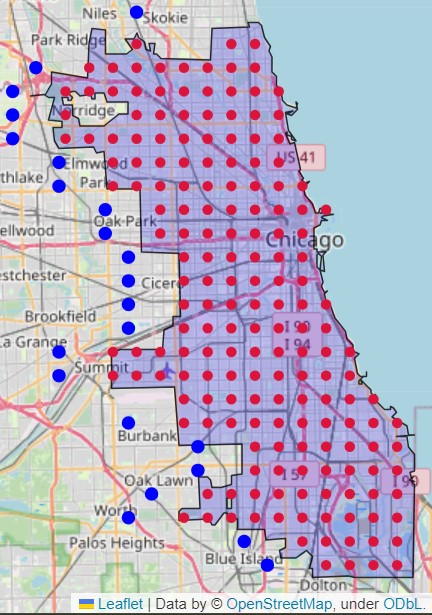}
         \caption{Chicago}
         \label{fig:Chicago}
     \end{subfigure}
     \hspace{3 pt}
     \begin{subfigure}[b]{0.462\textwidth}
         \centering
         \includegraphics[width=\textwidth]{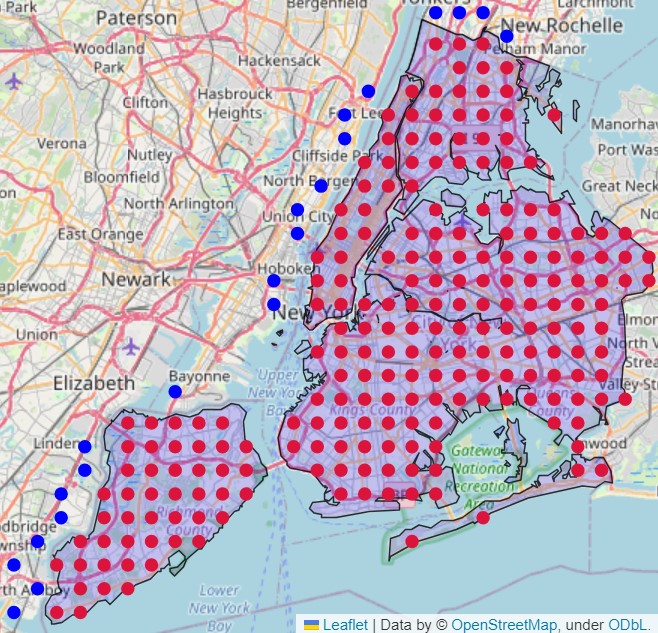}
         \caption{New York City}
         \label{fig:NYC}
     \end{subfigure}
        \caption{Microhubs (red markers) Located 1.5 Miles apart and Potential Depot Locations (blue markers)}
        \label{fig:Demand Centers}
\end{figure}

\begin{figure}
     \centering
     \begin{subfigure}[b]{0.4\textwidth}
         \centering
         \includegraphics[width=\textwidth]{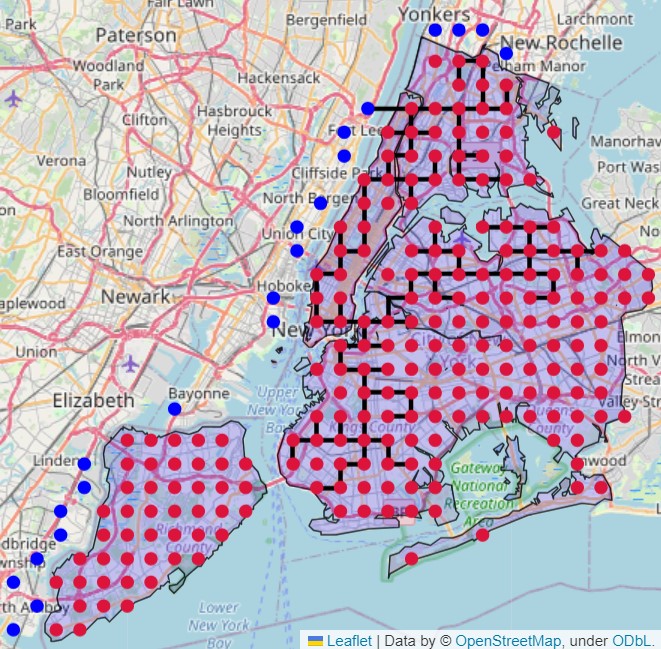}
         \caption{Arc Solution}
         \label{fig:ArcExample}
     \end{subfigure}
     \hspace{3 pt}
     \begin{subfigure}[b]{0.4\textwidth}
         \centering
         \includegraphics[width=\textwidth]{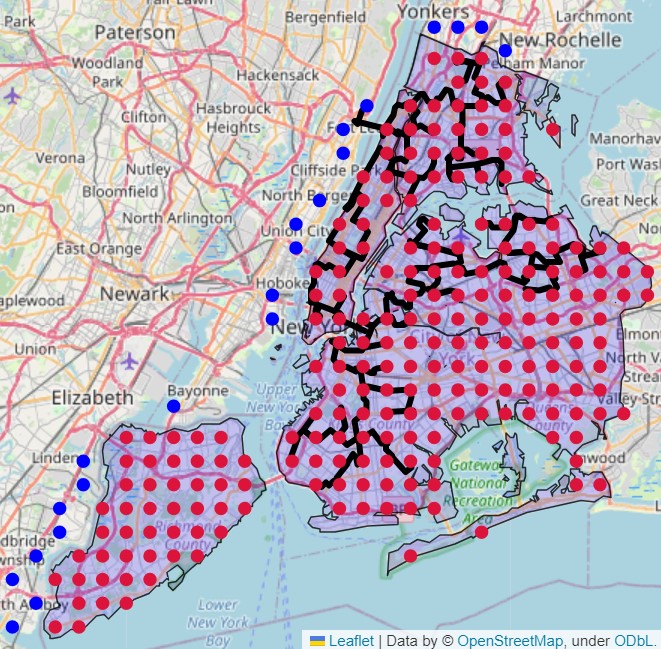}
         \caption{Road Network Solution }
         \label{fig:RoadExample}
     \end{subfigure}
        \caption{Example of Arc and Road Network Solutions in New York City}
        \label{fig:ArcSolutions}
\end{figure}

\subsubsection{Loading Capacity}
 The lead time required for the automated loading of carriages into the UFT system is 0.05 seconds per carriage. This equates to 1200 totes entering the network each minute. Each tote fits in a carriage. Based on data reported by the \cite{UKdft} and the volume of UFT totes of 11.65 $\text{ft}^3$, \cite{magway} estimates that between 80\% and 90\% of the goods that would typically be transported by road can fit into UFT totes.

Using a limit on the maximum loading time, $T$, at a depot and the lead time for each carriage, we define the maximum number of packages that can be delivered from a depot as $P_{max} = T/\tau$, where $\tau = 0.05/3600$, the lead time required for each carriage in hours.  We assume a loading time limit of $T = 12$ hours, giving $P_{max} = 864,000$. This allows time for carriages to return to the depot to be re-loaded for the next set of deliveries and for necessary maintenance on the carriages and UFT track.

\subsubsection{Demand at microhubs}

We test a set of demand levels at the microhubs to see how the effect of using a UFT network varies with different levels of usage. We determine the number of deliveries at each microhub using the population of the zip code in which the microhub is located (\cite{ZipCodeBoundaries}, \cite{PopulationData}, \cite{NYCPopulationData}) and the percentage of the population that has UFT delivery demand. To assign demand to each microhub, we divide a given percentage of the population in a zipcode evenly among the microhubs in the zip code.  We test UFT demand values of $10\%$, $15\%$, and $20\%$ of the population.  We center these values at $15\%$ of the population since \cite{haag20191} found that approximately $15\%$ of New York City households receive a package on any given day. We assume that a similar percentage of customers receive packages daily in Chicago as well. However, UFT delivery demand may vary due to the reluctance of service providers to move some or all of their deliveries to the new system or the fact that some packages cannot be transported via tunnels due to size or other factors. Also, as the volume of e-commerce orders increases, the percentage of households that receive a package each day will increase as well.  

There are exceptions that should be noted. For zip codes that do not contain a microhub, their demand is added to the nearest microhub. We determine the nearest microhub by calculating the distance from the centroid of the zip code without a microhub to the other microhubs. Figure \ref{fig:FixedDemand} shows an example of demand assigned to a set of microhubs where some zip codes do not have microhubs located within them.  These zip codes without microhubs are indicated by the dark red color, and the demand associated with these zip codes is assigned to the nearest microhub for each zip code. The maximum and minimum demand values at microhubs for each level of UFT demand are given in Table \ref{Tab:Demand}. This table shows the wide range of demand at different microhubs in a city, which will have a big impact on the resulting UFT networks. Alternative ways to distribute the demand in the red zip codes, as well as capacity limits on microhubs, are options for future work.

\begin{table}[h!] 
\begin{center}
\begin{tabular}{ |c|c|c|c|r| } 
 \hline
 & {$\%$ of} & {Min Demand} & {Max Demand} & {Total}\\
 {City}& {Population} & {at a Microhub} & {at a Microhub} & {Demand} \\
 \hline
  & 10\% & 215 & 10,893 & 297,338 \\ 
  Chicago & 15\% & 322 & 16,340 & 446,007 \\
  & 20\% & 429 & 21,787 & 594,676 \\
  \hline
  & 10\% & 395 & 18,080 & 838,969 \\
 New York City & 15\% & 592 & 27,120 & 1,258,454 \\
  & 20\% & 789 & 36,160 & 1,677,939 \\
 \hline
\end{tabular}
\caption{Demand Levels for Chicago and New York City} \label{Tab:Demand}
\end{center}
\end{table}

\subsubsection{Distance budgets}

In our set of experiments, we consider budgets on the length of tunnels built between 15 and 270 miles for Chicago and between 15 and 360 miles for New York City at increments of 15 miles for both cities. The larger range of budgets tested in New York City reflects the larger number of tunnels required to include all potential microhub locations in the UFT network.

\subsubsection{Performance measures}\label{measures}

To compute the measures representing savings from traditional delivery models,  we will compare the use of a UFT network to a collaborative delivery setting utilizing  UCCs.  These UCCs can use their own resources to consolidate and deliver packages to an urban area in an efficient manner. This is similar to the collaborative setting considered in \cite{handoko2014auction} and discussed in Section \ref{collaboration}. Collaborative delivery from UCCs requires a smaller number of trucks  and a lower cost of truck delivery from a depot to microhubs located within an urban area compared to several vendors making their own deliveries. We make this assumption to give a conservative estimate of the number of trucks, truck miles, and emissions saved by using a UFT network to make deliveries to a set of microhubs.

The number of trucks needed to serve a microhub is calculated from the daily demand at the microhub divided by the number of totes that can be carried in a large delivery truck. From \cite{TRL}, 960 totes can be delivered using a large delivery truck (\cite{magway}). We then sum over the microhubs included in the UFT network to determine the total number of truck trips saved for a range of budgets on tunnel length. The truck miles saved at a microhub are calculated as the number of trucks needed to serve the microhub multiplied by twice the road distance from the depot to the microhub. We multiply the distance to the microhub by two to account for the return trip to the depot. The miles saved are summed over all the microhubs that are included in the UFT network.

To estimate the reduction in emissions that will result from using a UFT network, we estimate the emissions that would be produced by serving the set of microhubs included in the UFT network by large trucks or delivery vans. For large delivery trucks, the \cite{EnvironmentalDefenseFund} estimates $\text{CO}_2$ emissions of 1,700 grams per mile in their Green Freight Handbook.  Delivery vans generate 248 grams of $\text{CO}_2$ emissions per mile \citep{TransportAndEnvironment}. For delivery by large truck, we use the same number of trucks and truck miles estimated as described above. For delivery vans, we use the estimate from \cite{amazon} that about 350 packages can fit in a van to determine the number of vans needed for each microhub. We then calculate the number of miles traveled by vans to serve a microhub by multiplying the number of vans needed by twice the road distance between the depot and the microhub.

\begin{figure}[!ht]
    \centering
    \includegraphics[width = 6cm]{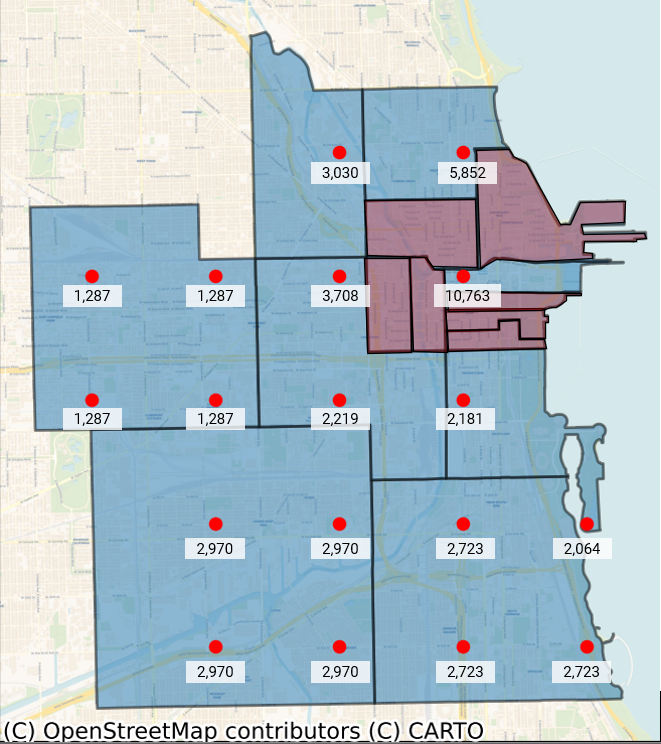}
    \caption{Equally divided demand based on $15\%$ of Zip Code Population}
    \label{fig:FixedDemand}
\end{figure}

\subsection{Performance of Decomposition} \label{method}
{We demonstrate the performance of our solution methodology for both the Chicago and New York City instances described in Section \ref{ExperimentalDesign.Network} at the $15\%$ demand level. The run time for the mixed integer multicommodity flow formulation from Section \ref{Model1} and the decomposition approach from Section \ref{sec:decomp} are given in Tables \ref{Table: DirectedChicago} and \ref{Table: DirectedNYC}. The results for the multicommodity flow model are labeled MIP and our decomposition approach are labeled MP. We solve the multicommmodity flow formulation using Gurobi, and implement MP using a Gurobi call back function, both of which are coded in Python. We enforced a two hour time limit on all instances and denote that the runtime is $> 7200$ for instances that could not be solved to optimality within this limit.

For instances in Chicago, the results in Table \ref{Table: DirectedChicago} demonstrate that MP outperforms the MIP formulation across all instances, significantly reducing the run time. The longest MIP run times occur for values of $D$ greater than 240 in both the $n_d = 1$ and $n_d = 2$ cases, with a maximum run time over 700 seconds when $n_d = 1$ and as large as 404 seconds when $n_d = 2$. But, MP is able to solve these instances to optimiality in less than one second. 

For New York City instances, in Table \ref{Table: DirectedNYC},  when $n_d = 1$, MP provides shorter running times for most instances, with the exception of the case where $D = 105$. For this problem, the running time of the MIP and MP are similar, with a difference of less than seven seconds. When $D \geq 135$, the MIP becomes significantly harder to solve, with running times as large as 5,007 seconds. For all of the instances with $D \geq 135$, MP was able to reach an optimal solution in less than two seconds.

When $n_d = 2$ for New York City instances, Table \ref{Table: DirectedNYC} indicates MP has shorter running times than the MIP for cases with $D < 240$. For this range of tunnel budgets, MP has a maximum running time of approximately 29 seconds while the MIP has a maximum running time of approximately 136 seconds. But, when $D > 240$, the problem becomes much more challenging to solve for both the MIP and MP. For these instances, the tunnel budget is large enough that the depot capacity becomes binding. When demand is distributed such that there are sets of high demand nodes clustered together and located far enough away from the set of potential depot locations that it is not possible to cover this set of demand nodes using two depots with the given tunnel budget, it becomes challenging to find solutions that respect capacity constraints  while serving as many microhubs as possible in the high demand area. For tunnel budgets between $D=240$ and $D= 285$, MP is unable to find an optimal solution in less than two hours, but provides an improvement to the MIP for tunnel budgets $D\geq 300$. }

\begin{table}
\begin{center}
\begin{tabular}{ |c|c|c|c||c|c|c| } 
 \hline
  & & MIP & MP & & MIP & MP \\ 
 $d$ & $n_d$ & Run Time (s) & Run Time (s) & $n_d$ & Run Time (s) & Run Time (s)\\
 \hline
 15 & 1 & 12.11 & 1.27 & 2 & 21.47 & 1.19 \\
 \hline
 30 & 1 & 6.07 & 0.99 & 2 & 4.80 & 0.95\\ 
 \hline
 45 & 1 & 4.78 & 0.87 & 2 & 5.45 & 0.72 \\
 \hline
 60 & 1 & 4.22 & 0.40 & 2 & 3.77 & 0.42 \\
 \hline 
 75 & 1 & 4.48 & 0.59 & 2 & 5.45 & 0.56 \\
 \hline
 90 & 1 & 7.55 & 0.65 & 2 & 4.42 & 0.58 \\
 \hline
 105 & 1 & 9.02 & 1.86 & 2 & 7.36 & 1.42 \\
 \hline
 120 & 1 & 17.25 & 1.19 & 2 & 16.81 & 1.28\\
 \hline
 135 & 1 & 18.60 & 0.53 & 2 & 12.39 & 0.82 \\
 \hline
 150 & 1 & 8.03 & 0.46 & 2 & 8.77 & 0.49 \\
 \hline
 165 & 1 & 19.89 & 0.68 & 2 & 10.79 & 0.49 \\
 \hline
 180 & 1 & 11.09 & 0.39 & 2 & 8.55 & 0.56 \\
 \hline
 195 & 1 & 10.78 & 0.58 & 2 & 18.06 & 0.93 \\
 \hline
 210 & 1 & 48.09 & 0.88 & 2 & 125.30 & 0.82  \\
 \hline
 225 & 1 & 6.53 & 0.50 & 2 & 36.00 & 0.81 \\
 \hline
 240 & 1 & 463.61 & 0.50 & 2 & 5.99 & 0.96 \\
 \hline
 255 & 1 & 718.23 & 0.30 & 2 & 404.32 & 0.41 \\
 \hline
 270 & 1 & 468.53 & 0.31 & 2 & 267.04 & 0.45 \\
 \hline
\end{tabular}
\caption{Results for $n_d = 1$ and $n_d = 2$ for Chicago Instances at a $15\%$ Demand Level}
\label{Table: DirectedChicago}
\end{center}
\end{table}

\begin{table}
\begin{center}
\begin{tabular}{ |c|c|c|c||c|c|c| } 
 \hline
  & & MIP & MP & & MIP & MP \\ 
 $d$ & $n_d$ & Run Time (s) & Run Time (s) & $n_d$ & Run Time (s) & Run Time (s)\\
 \hline
 15 & 1 & 25.48 & 0.96 & 2 & 24.70 &  0.22 \\
 \hline
 30 & 1 & 22.22 & 1.77 & 2 & 20.00 &  0.71 \\ 
 \hline
 45 & 1 & 21.08 & 13.50 & 2 & 17.72 &  1.61 \\
 \hline
 60 & 1 & 13.69 & 9.42 & 2 & 17.95 &  3.88  \\
 \hline 
 75 & 1 & 30.04 & 11.15 & 2 & 17.24 &  2.63 \\
 \hline
 90 & 1 & 21.99 & 7.92 & 2 & 17.08 &  1.88 \\
 \hline
 105 & 1 & 15.58 & 22.10 & 2 & 15.90 &  6.97 \\
 \hline
 120 & 1 & 24.65 & 21.89 & 2 & 18.83 &  5.29 \\
 \hline
 135 & 1 & 5007.03 & 1.96 & 2 & 17.54 &  6.11 \\
 \hline
 150 & 1 & 3334.96 & 1.00 & 2 & 16.86 &  7.57 \\
 \hline
 165 & 1 & 4809.79 & 0.53 & 2 & 19.79 & 8.53 \\
 \hline
 180 & 1 & 1578.84 & 0.20 & 2 & 16.80 & 2.22 \\
 \hline
 195 & 1 & 4194.64 & 0.07 & 2 & 16.98 &  4.01 \\
 \hline
 210 & 1 & 1586.15 & 0.63 & 2 & 30.67 & 3.09 \\
 \hline
 225 & 1 & 939.77 & 0.12 & 2 & 135.91 & 28.64 \\
 \hline
 240 & 1 & 356.84 & 0.10 & 2 & $>$ 7200 & $>$ 7200 \\
 \hline
 255 & 1 & 275.57 & 0.10 & 2 & 6042.32 & $>$ 7200 \\
 \hline
 270 & 1 & 1423.43 & 0.07 & 2 & 5091.68 & $>$ 7200 \\
 \hline 
 285 & 1 & 954.49 & 0.14 & 2 & 5415.18 & $>$ 7200\\
 \hline
 300 & 1 & 452.55 & 0.10 & 2 & 6667.37 & 1665.61 \\
 \hline 
 315 & 1 & 524.25 & 0.14 & 2 & $>$ 7200 & 2638.33 \\
 \hline
 330 & 1 & 614.13 & 0.13 & 2 & 5688.09 & 1622.40 \\
 \hline
 345 & 1 & 704.38 & 0.11 & 2 & 4027.40 & 3096.31 \\
 \hline
 360 & 1 & 420.28 & 0.11 & 2 & $>$ 7200 & 689.55 \\
 \hline
\end{tabular}
\caption{Results for $n_d = 1$ and $n_d = 2$ for New York City Instances at a $15\%$ Demand Level}
\label{Table: DirectedNYC}
\end{center}
\end{table}

\subsection{Performance Metrics for UFT Networks} \label{results}
Here we provide results for the amount of demand that can be served by UFT systems, the number of trucks and miles saved daily, and the environmental impact of using a UFT system. 

\subsubsection{Demand coverage}

We evaluate the percentage of total microhub demand served by the UFT network for a range of tunnel budgets for both Chicago and New York City. In Figure \ref{fig:Demand1Depot}, we present the percentage of demand served by the UFT network for each level of demand when $n_d = 1$. In both Chicago and New York City, the demand served by the UFT network increases rapidly with the tunnel budget when $D$ is small. With only 45 miles of tunnels, approximately $42\%$ of package demand in Chicago and approximately $32\%$ of package demand in New York City can be served with a UFT network. This reduces the number of packages served via traditional delivery methods to 58\% and 68\% of their original values. Almost all microhub demand can be served by 250 miles of tunnels in Chicago and 350 miles in New York City (with demand at 10\% of population). 

\noindent\textit{Insight 1:  Low tunnel mileage can potentially create a tremendous reduction in the flow of packages via traditional delivery methods. }

An example demonstrating how different tunnel budgets are used is provided in Figure \ref{fig:ChicagoExamplesStructure} for the city of Chicago. When the tunnel budget is small, such as 75 miles in Figure \ref{fig:Chicago75},  the arcs selected for the tunnel connect a depot to downtown Chicago and cover the high demand microhubs located in the northern part of the city and along the lake, serving approximately $50\%$ of delivery demand. As the tunnel budget increases (see Figures \ref{fig:Chicago150} and \ref{fig:Chicago225}), these high demand microhubs continue to be included in the solution along with other lower demand microhubs. We note that the location of the selected depot also changes as the tunnel budget increases.

\begin{figure}
     \centering
     \begin{subfigure}[b]{0.49\textwidth}
    \centering
    \includegraphics[width = 1\textwidth]{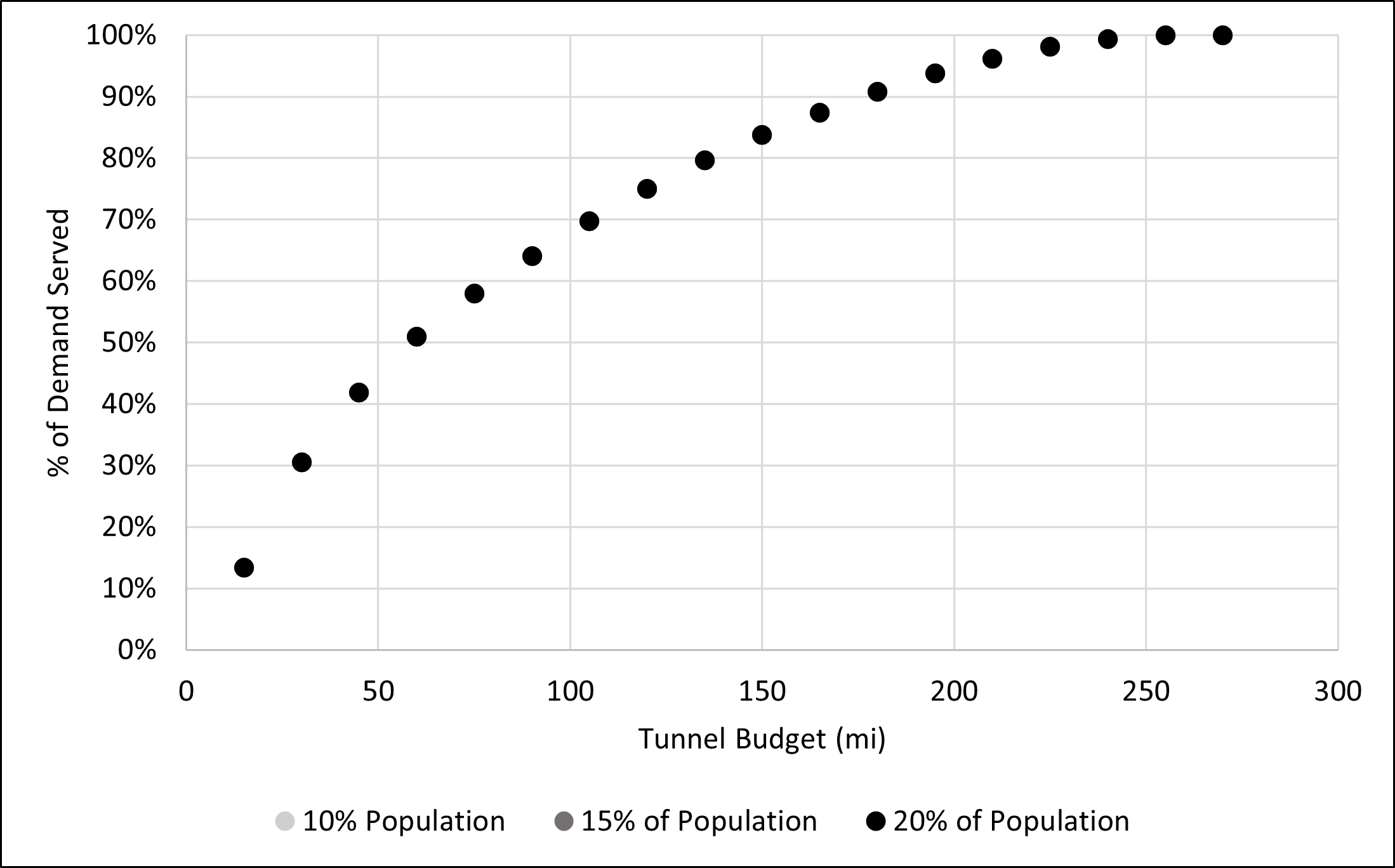}
    \caption{Chicago}
    \label{fig:ChicagoDemand1Depot}
     \end{subfigure}
     \hfill
     \begin{subfigure}[b]{0.49\textwidth}
         \centering
    \includegraphics[width = 1\textwidth]{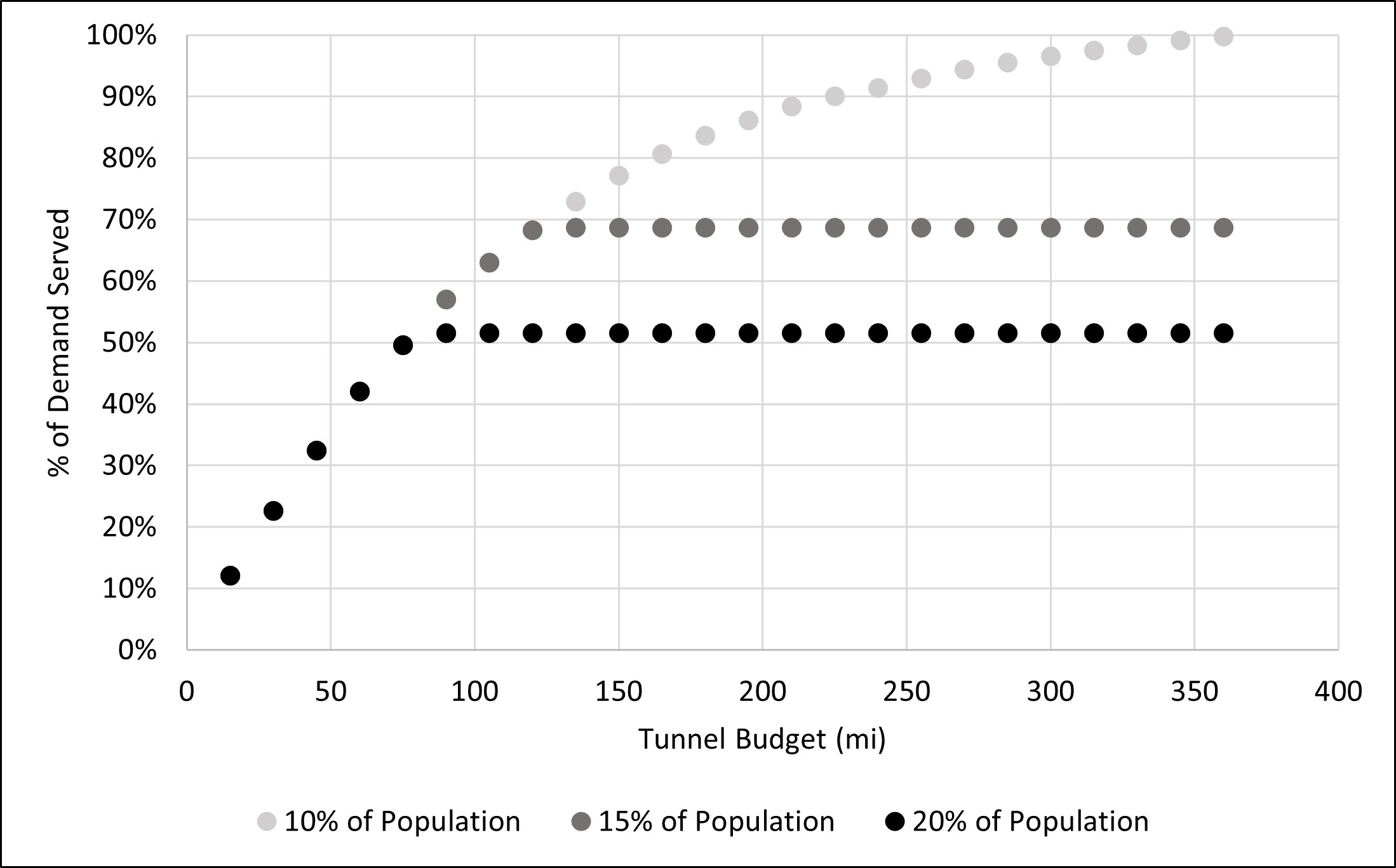}
    \caption{New York City}
    \label{fig:NewYorkDemand1Depot}
     \end{subfigure}
        \caption{Percent of Demand Served for the 1-UFT}
        \label{fig:Demand1Depot}
\end{figure}

\begin{figure}
     \centering
     \begin{subfigure}[b]{0.317\textwidth}
         \centering
    \includegraphics[width = 1\textwidth]{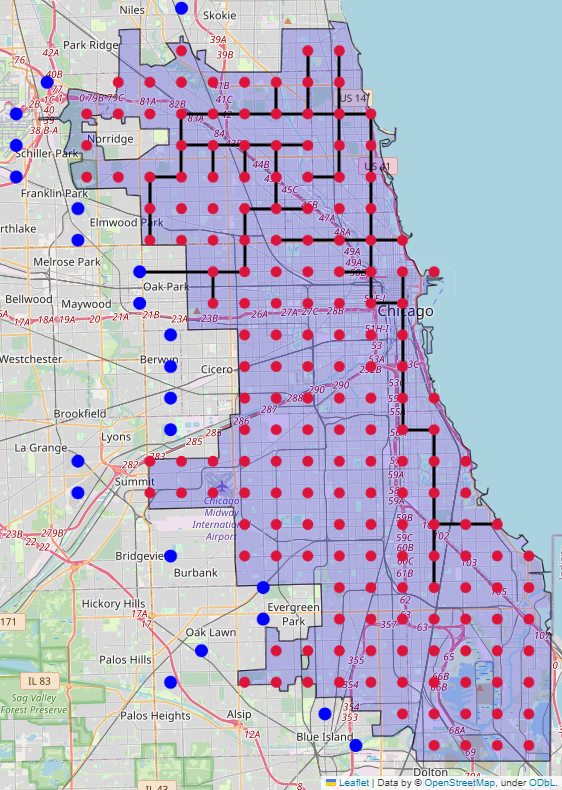}
    \caption{Tunnel Budget of 75 Miles}
    \label{fig:Chicago75}
     \end{subfigure}
     \hfill
     \begin{subfigure}[b]{0.315\textwidth}
         \centering
    \includegraphics[width = 1\textwidth]{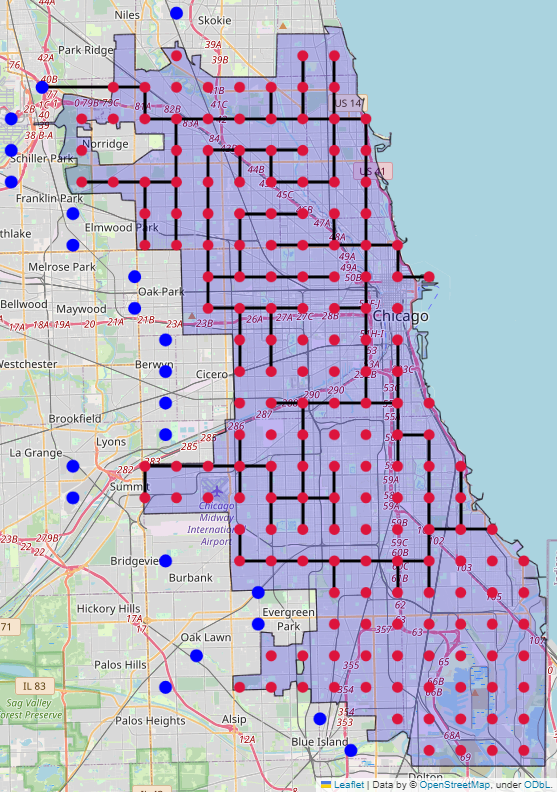}
    \caption{Tunnel Budget of 150 Miles}
    \label{fig:Chicago150}
     \end{subfigure}
     \hfill
     \begin{subfigure}[b]{0.315\textwidth}
         \centering
    \includegraphics[width = 1\textwidth]{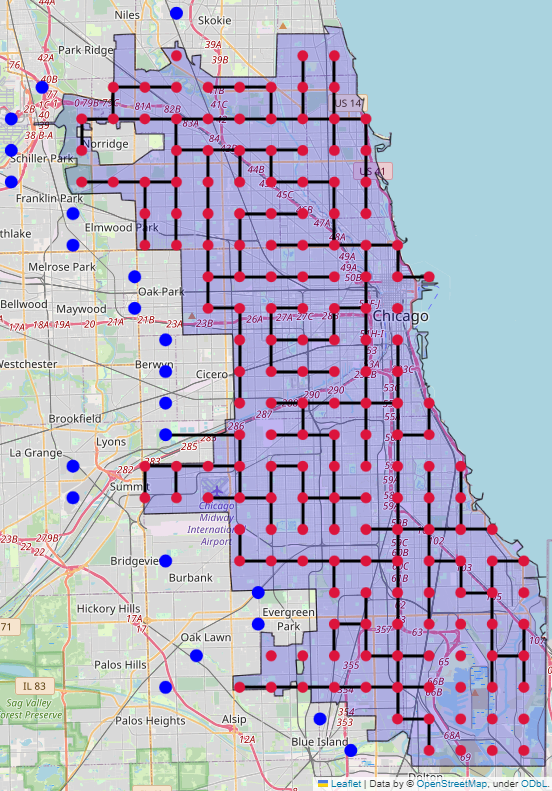}
    \caption{Tunnel Budget of 225 Miles}
    \label{fig:Chicago225}
     \end{subfigure}
    \caption{Chicago Solutions with 20\% Demand Levels and $n_d = 1$ }
    \label{fig:ChicagoExamplesStructure}
\end{figure}

For Chicago, in Figure \ref{fig:ChicagoDemand1Depot}, for all levels of demand, the percentage of total demand served is the same at each tunnel budget. This is because all of the demand in Chicago can be served by one depot within the 12 hour loading time limit when demand is calculated as $10\%$, $15\%$, or $20\%$ of the population. But, when the depot loading capacity is binding, as it is with New York City, the amount of demand served by the UFT network will vary with demand level. This can be seen in the results for New York City in Figure \ref{fig:NewYorkDemand1Depot}. When demand is $10\%$ of the population, then the depot loading capacity is not binding for any tunnel budget, and the percent of demand served follows a similar curve to that in Chicago. When the level of UFT demand is $15\%$ of the population, at most $69\%$ of demand can be served from one depot since the depot loading capacity becomes binding for tunnel budgets greater than 120 miles. Similarly, when a demand level of $20\%$ of the population is considered, the loading capacity becomes binding when the tunnel budget is greater than 90 miles. When the tunnel budget is less than 90 miles, the same percentage of demand is served by the UFT network as in $10\%$ and $15\%$ demand levels. But, when the tunnel budget exceeds 90 miles, the depot loading capacity is reached, and no more than $51\%$ of demand can be served by the UFT network.

\noindent \textit{Insight 2: When the depot loading capacity becomes binding, there are no further increases in total demand served even with increasing tunnel budgets. }

Increasing the number of depots allows a larger percentage of demand to be served by the UFT system since the loading capacity increases by $P_{max}$ for each depot added. In Figure \ref{fig:NewYorkDemand20}, we show the percentage of UFT demand that is served by the UFT network when one and two depots are used in New York City, and demand is $20\%$ of the population. We only show results for New York City since the depot loading capacity is not binding in Chicago. For New York City, using two depots allows more demand to be served by UFT at all tunnel budgets considered, but especially for tunnel budgets over 90 miles as shown in Figure \ref{fig:NewYorkDemand20}. When the tunnel budget is less than 90 miles, slightly larger percentages of total demand can be served when two depots are used due to the ability to choose two depots near different areas of high-demand microhubs. 
And, when the tunnel budget is over 90 miles, the percentage of demand served continues to increase when two depots are used since the loading capacity of the UFT network is now large enough to serve all demand in New York City at a $20\%$ demand level. This creates an interesting change in the structure of the solutions. 
When two depots are used in New York City with a large tunnel budget,  more high demand nodes can be included in the solution (in Figure \ref{fig:NYC.150.2depots}) than with one depot (in Figure \ref{fig:NYC.150.1depot}), due to the added loading capacity. The two depots selected in Figure \ref{fig:NYC.150.2depots} are near the high demand areas of Manhattan and the Bronx.  

\noindent \textit{Insight 3: When the depot loading capacity is binding, adding a depot improves solutions in terms of the percentage of demand served by a UFT system.}

\begin{figure}
    \centering
    \includegraphics[width = 0.6\textwidth]{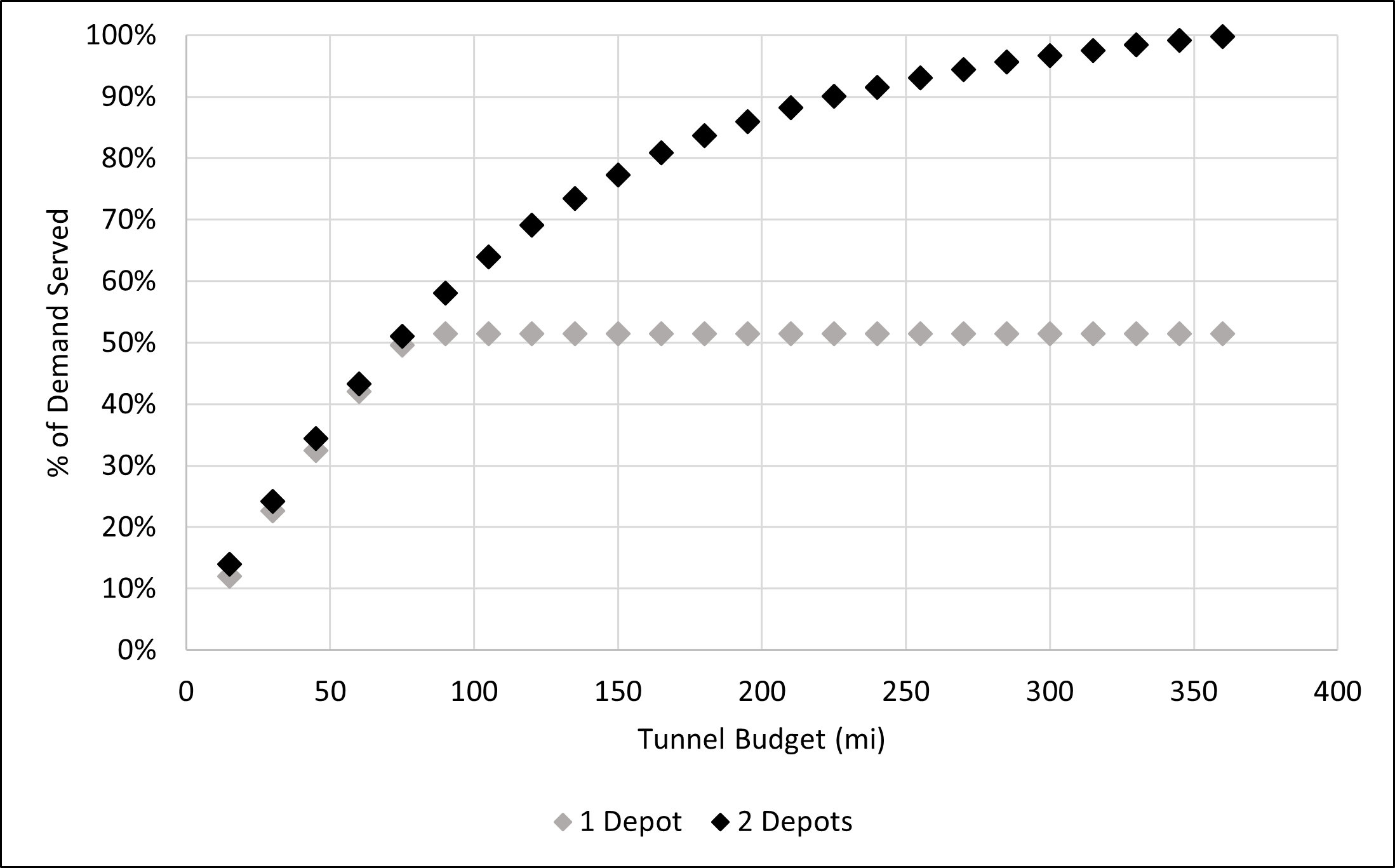}
    \caption{Percent of Demand Served in New York City at a Demand Level of $20\%$}
     \label{fig:NewYorkDemand20}
\end{figure}

 \begin{figure}
     \centering
     \begin{subfigure}[b]{0.49\textwidth}
         \centering
    \includegraphics[width = 1\textwidth]{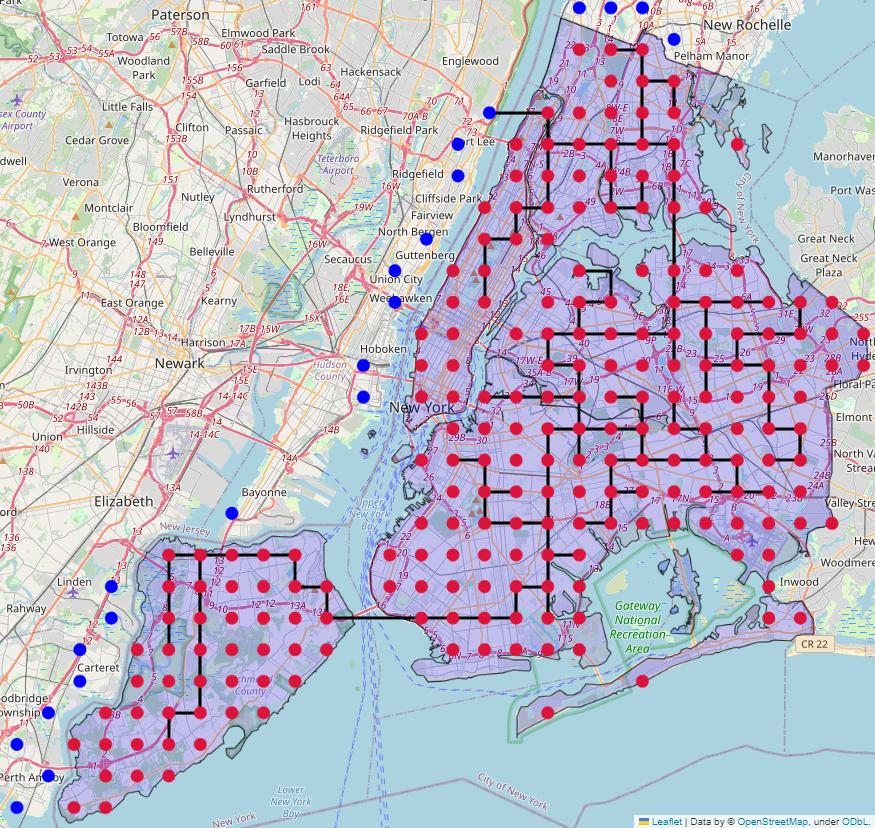}
    \caption{1 Depot}
    \label{fig:NYC.150.1depot}
     \end{subfigure}
     \hfill
     \begin{subfigure}[b]{0.49\textwidth}
         \centering
    \includegraphics[width = 1\textwidth]{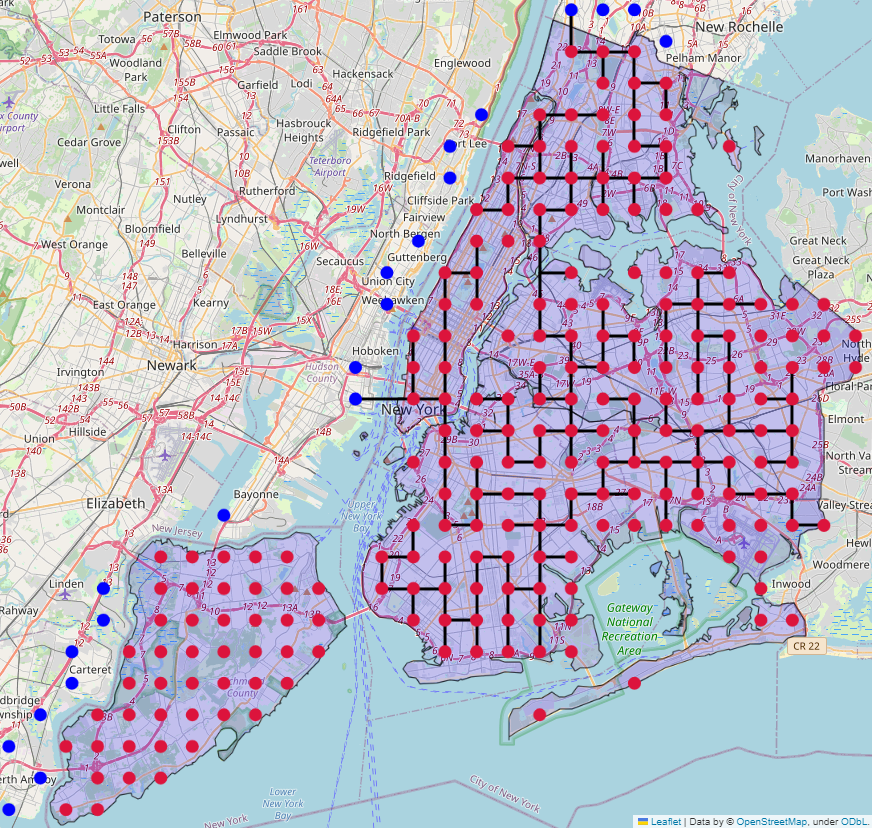}
    \caption{2 Depots}
    \label{fig:NYC.150.2depots}
     \end{subfigure}
    \caption{New York City Solutions with 20\% Demand Levels and $d = 225$}
    \label{fig:NYCExamplesStructureDepots}
\end{figure}

\subsubsection{Truck trips}
Since one of the primary advantages of a UFT network is the reduction of the number of delivery trucks on the road, it is important to estimate the number of truck trips saved daily by the use of a UFT system.  As noted in Section \ref{measures}, we compute these savings compared to a collaborative delivery setting.  The results for a range of tunnel budgets and demand levels when $n_d = 1$  are illustrated in Figure \ref{fig:Trucks1Depot}. The shape of the curve for the number of truck trips saved for both Chicago and the $10\%$ demand level in New York City are similar to the curve for the percentage of demand served for these instances. For Chicago, since the depot loading capacity is large enough to serve all demand, the number of truck trips saved increases with the demand level for all tunnel budgets with 49-393 trips saved daily at a 10\% demand level and 88-718 trips saved daily at a 20\% demand level depending on the value of $d$. For New York City, when the tunnel budget is less than 90 miles, we see a similar trend with a larger number of truck trips saved for higher levels of UFT demand. But as the tunnel budget increases and the depot loading capacity is reached for demand levels of $15\%$ and $20\%$, the savings are no longer increasing. When the demand level is $10\%$ of the population, and the depot loading capacity is not constraining, a maximum of 984 truck trips are saved when all demand is served by the UFT network. When the demand level is at $15\%$ of the population, and the tunnel budget is over 120 miles, the depot loading capacity is reached, and the number of trucks saved stabilizes at 957 truck trips. Even with loading capacity restrictions, the numbers for New York City are larger than Chicago due to the higher population and thus number of packages.

\noindent \textit{Insight 4:  A UFT can save hundreds of truck trips daily, with the precise amount depending on the tunnel budget, demand at microhubs/population, geography, and depot loading capacity.}

\noindent \textit{Insight 5: When the depot loading capacity is binding, then the daily number of truck trips saved will stabilize at a given tunnel budget.}

\begin{figure}
     \centering
     \begin{subfigure}[b]{0.49\textwidth}
         \centering
    \includegraphics[width = 1\textwidth]{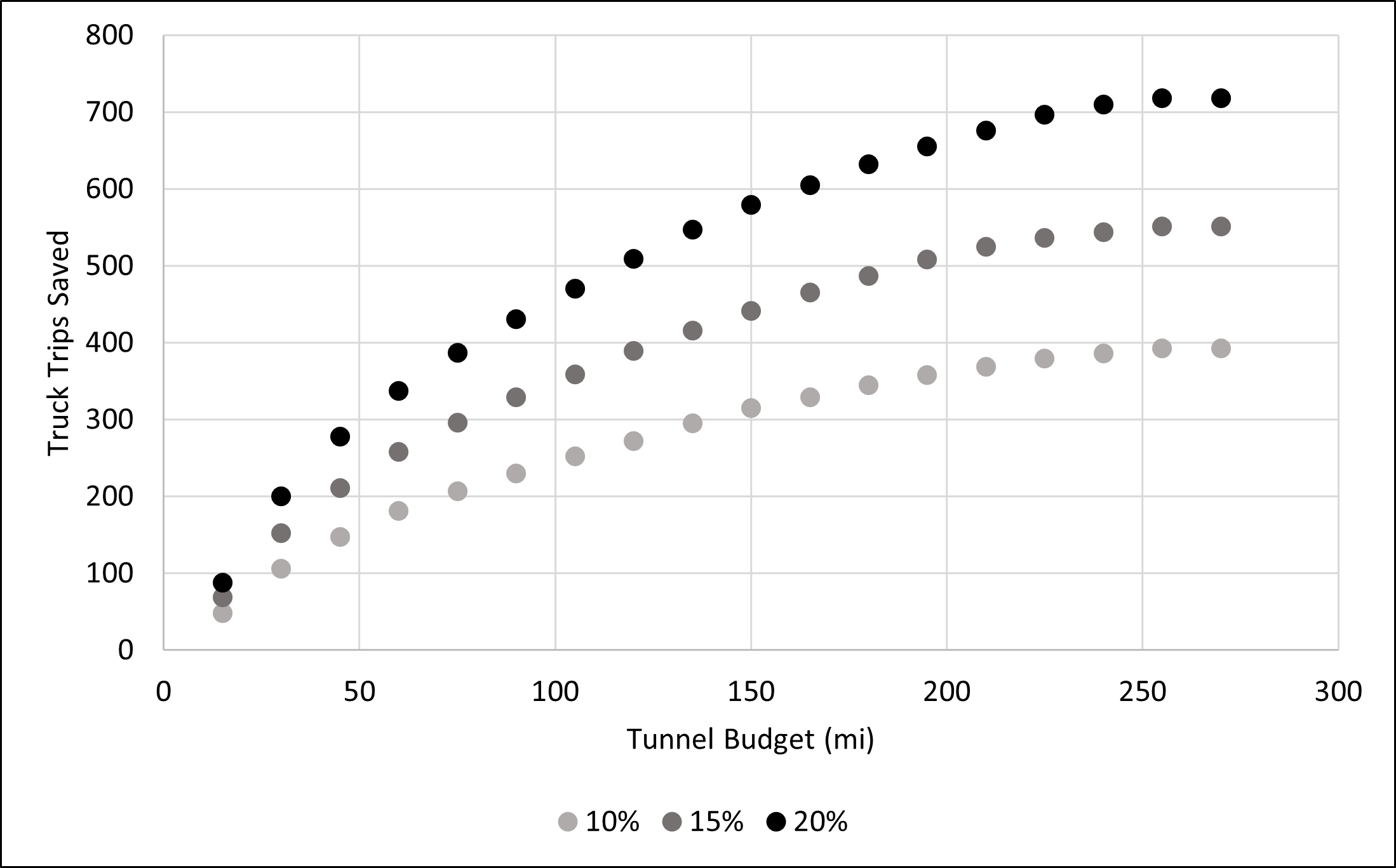}
    \caption{Chicago}
    \label{fig:ChicagoTrucks1Depot}
     \end{subfigure}
     \hfill
     \begin{subfigure}[b]{0.49\textwidth}
         \centering
    \includegraphics[width = 1\textwidth]{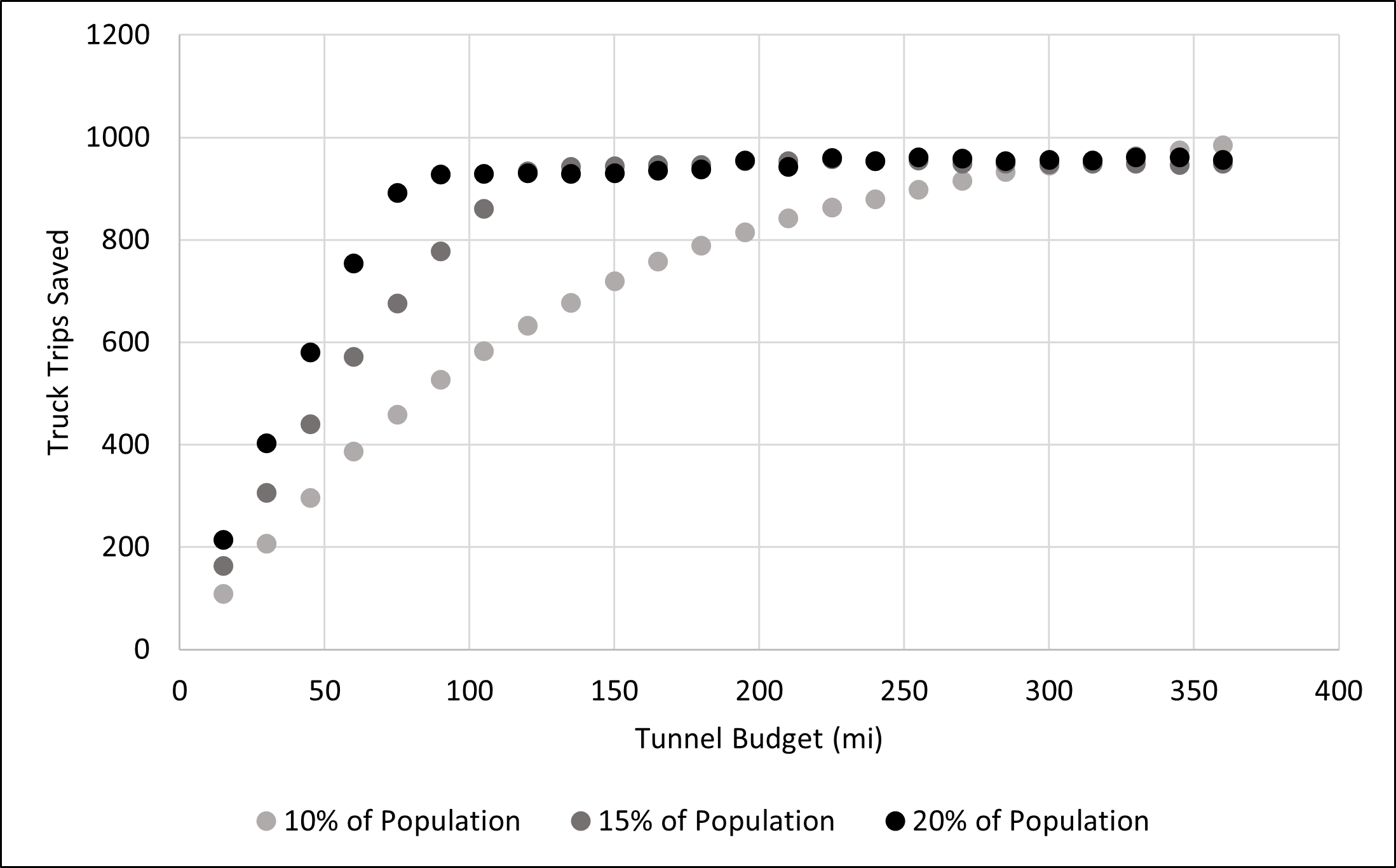}
    \caption{New York City}
    \label{fig:NewYorkTrucks1Depot}
     \end{subfigure}
        \caption{Number of Truck Trips Saved Daily by the 1-UFT}
        \label{fig:Trucks1Depot}
\end{figure}

The number of trucks saved in New York with $20\%$ demand density with one and two depots are shown in Figure \ref{fig:NewYorkTrucks20}. As with the amount of demand served, we do not show results for two depots in Chicago since the results are the same when one or two depots are used. When the tunnel budget is greater than 90 miles and $n_d = 2$, the number of truck trips saved continues to increase at a rate similar to the percent of demand served from two depots in New York City, as shown in Figure \ref{fig:NewYorkTrucks20} and discussed above. With two depots, all UFT demand can be served, and at most 1,853 truck trips are saved when there is a demand level of $20\%$. 

\noindent \textit{Insight 6: When the depot loading capacity is binding, using a second depot allows a greater reduction in the number of truck trips needed to serve the city.}

\begin{figure}
    \centering
    \includegraphics[width = 0.6\textwidth]{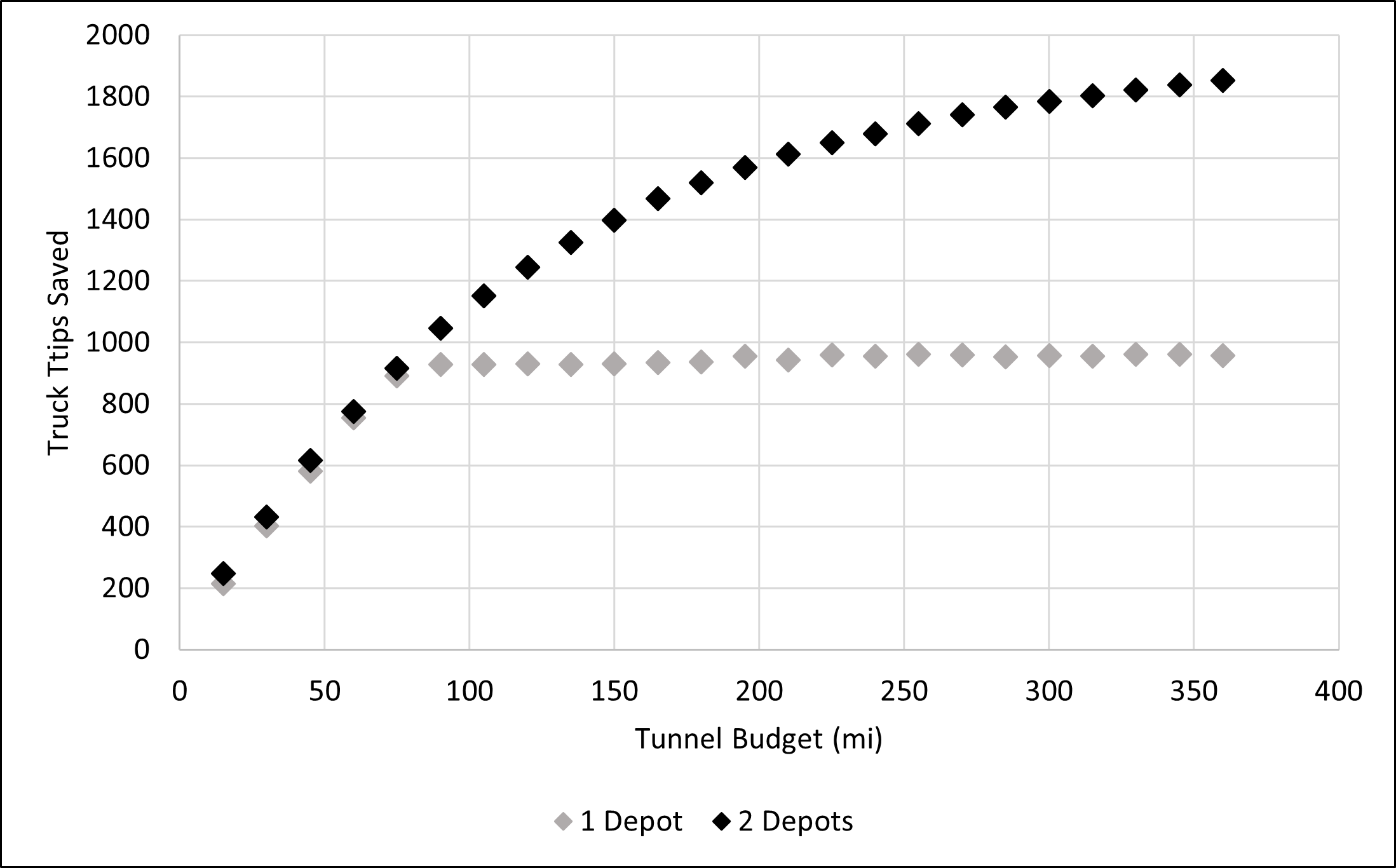}
    \caption{Number of Truck Trips Saved in New York City at a Demand Level of $20\%$}
    \label{fig:NewYorkTrucks20}
\end{figure}

\subsubsection{Truck miles}
We can estimate the number of truck miles saved using the method described in Section \ref{measures}. The results for Chicago for $n_d = 1$ and 20\% demand are in Figure \ref{fig:ChicagoTruckMiles20}. In Chicago, as for the number of truck trips saved, the number of truck miles saved at any given tunnel budget is generally increasing with the level of UFT demand, but the patterns are not as smooth as with the previous measures. There are tunnel budgets at which there is a sharp increase or decrease in the number of truck miles saved, as shown in Figure \ref{fig:ChicagoTruckMiles20}. One section with these jumps occurs at tunnel budgets between 120 and 150 miles. For a $20\%$ demand level, when the tunnel budget is 120 miles, less than 9,000 truck miles are saved, but at a tunnel budget of 135 miles, almost 14,000 truck miles are saved. The solutions for these budgets are shown in Figures \ref{fig:Chicago120} and \ref{fig:Chicago135}. For a tunnel budget of 135 miles, the depot is located farther north than for the 120 mile budget. This change in the network design makes minimal difference the objective of the UFT, but it makes a big change in the cost to serve the microhubs via trucks.  When modeling the distance via trucks, we have to consider multiple trips from the depot to high demand microhubs, which are now further from the depot in Figure \ref{fig:Chicago135}, and new long trips from the depot to the microhubs located in the southern part of the city in \ref{fig:Chicago135}.  
We can see this phenomenon again when we compare the case with a tunnel budget of 135 miles to at a tunnel budget of 165 miles in Figure \ref{fig:Chicago165}. Here the depot is located close to the central west side of Chicago, resulting in shorter individual trip distances between the depot and the microhubs. Even though the UFT does not consider truckload sized trips, certain hub locations using the same tunnel budget will require more total distance to be traveled by carriages and thus more energy usage, as discussed in detail in Section \ref{CostAnalysis}.  Thus, this may suggest potential further study of alternative or additional objective functions in future work.

 \begin{figure}
     \centering
     \begin{subfigure}[b]{0.49\textwidth}
         \centering
    \includegraphics[width = 1\textwidth]{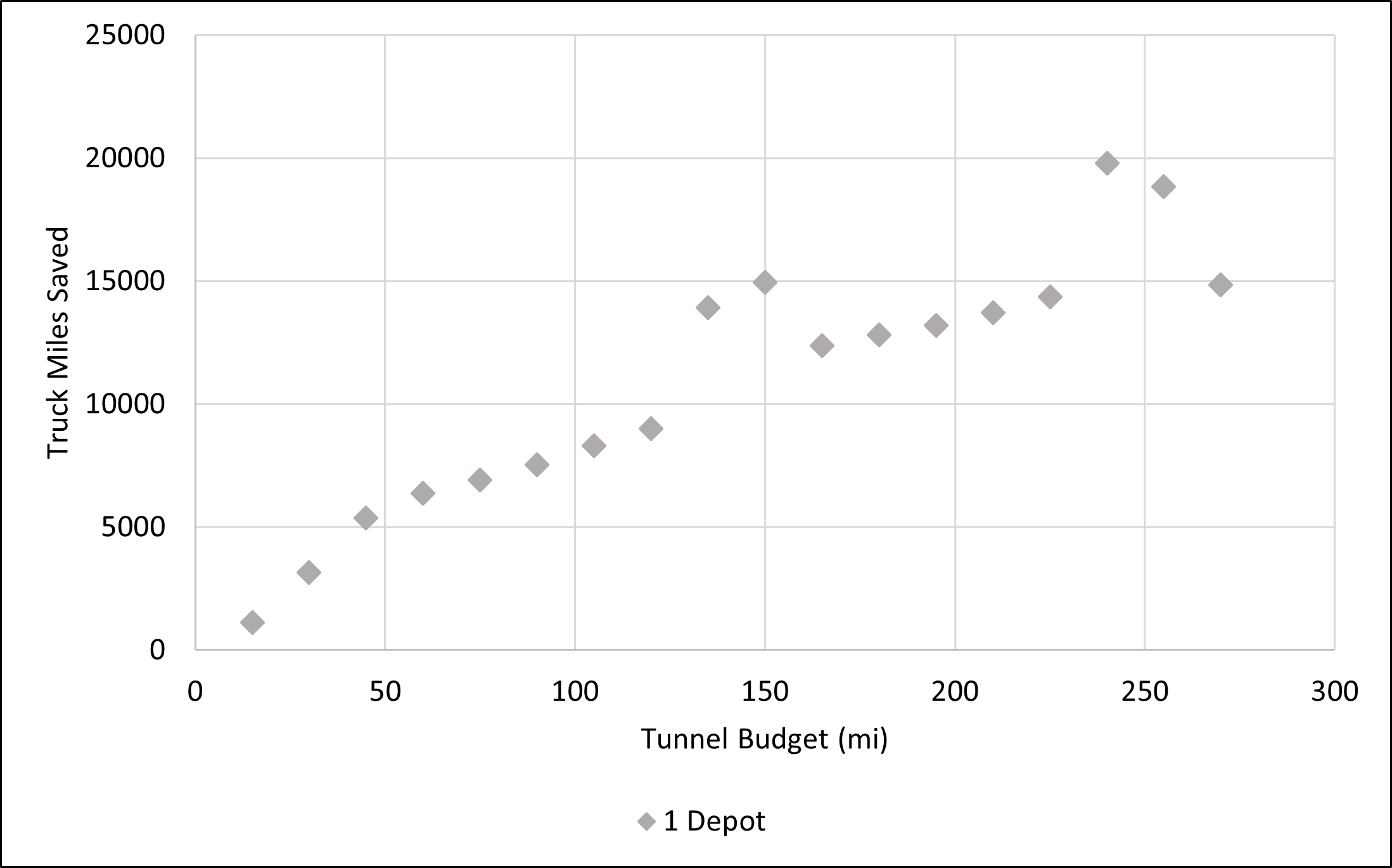}
    \caption{Chicago, 1 Depot}
    \label{fig:ChicagoTruckMiles20}
     \end{subfigure}
     \hfill
     \begin{subfigure}[b]{0.49\textwidth}
         \centering
    \includegraphics[width = 1\textwidth]{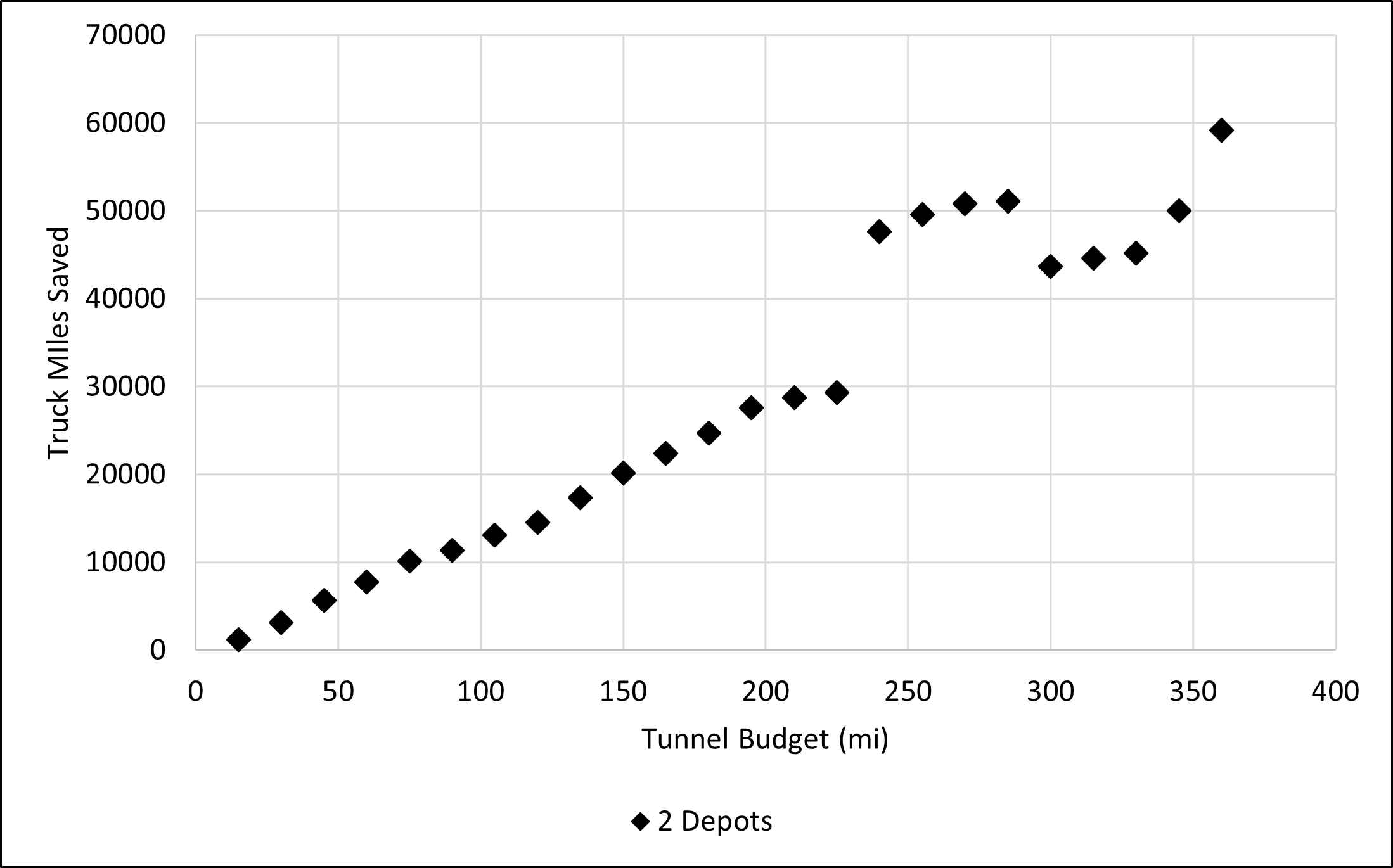}
    \caption{New York City, 2 Depots}
    \label{fig:NewYorkTruckMiles20}
     \end{subfigure}
        \caption{Number of Truck Miles Saved at a Demand Level of $20\%$}
        \label{fig:TruckMiles20}
\end{figure}

\noindent \textit{Insight 7: The selected depot location can have a significant impact on the number of truck miles saved by the use of a UFT system even when it has little impact on the UFT objective. }

\begin{figure}
     \centering
     \begin{subfigure}[b]{0.317\textwidth}
         \centering
    \includegraphics[width = 1\textwidth]{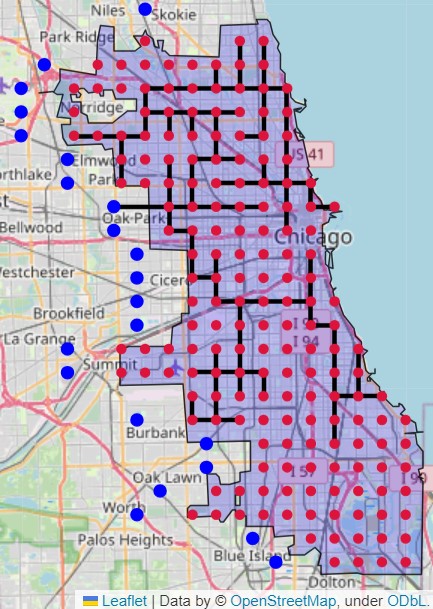}
    \caption{Tunnel Budget of 120 Miles}
    \label{fig:Chicago120}
     \end{subfigure}
     \hfill
     \begin{subfigure}[b]{0.315\textwidth}
         \centering
    \includegraphics[width = 1\textwidth]{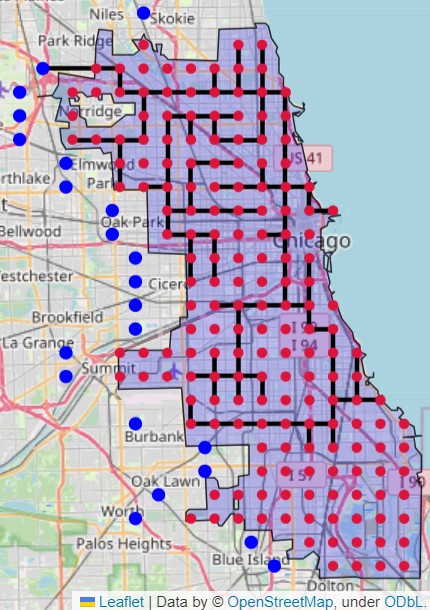}
    \caption{Tunnel Budget of 135 Miles}
    \label{fig:Chicago135}
     \end{subfigure}
     \hfill
     \begin{subfigure}[b]{0.308\textwidth}
         \centering
    \includegraphics[width = 1\textwidth]{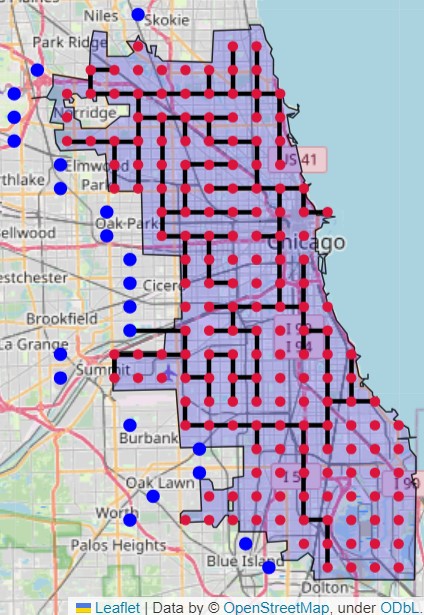}
    \caption{Tunnel Budget of 165 Miles}
    \label{fig:Chicago165}
     \end{subfigure}
    \caption{Chicago Solutions with 20\% Demand Levels and $n_d = 1$}
    \label{fig:ChicagoExamples1}
\end{figure}

For New York City, as shown in Figure \ref{fig:NewYorkTruckMiles20}, for $n_d = 2$ there is a steady increase in the number of miles saved as the tunnel budget increases while it is less than 225 miles. There is a large jump in the number of truck miles saved from a tunnel budget of 225 miles to 240 miles from 29,331 to 47,656 truck miles saved due to the location of the depots relative to the microhubs assigned to them. And the number of truck miles saved decreases as the tunnel budget increases from 285 miles to 300 miles due to smaller distances between microhubs and the depot they are assigned to, similar to the case for Chicago discussed above.

\section{Cost analysis} \label{CostAnalysis}

The above results indicate that a UFT system could have a big impact on taking trucks off the roads as well as reducing emissions.  Thus, the obvious next question is how much would a UFT actually cost to build and operate? And how much could cities recoup through charging for the use of such a system?  The latter question is a particularly challenging one given all of the uncertainties. There is uncertainty about how much savings that different delivery service providers would estimate from using a UFT, their willingness for change, what interventions and incentives a city would use to encourage the use of UFTs, etc.  Thus, we defer the discussion of pricing, and sensitivity to pricing levels, to future work.  We will try to estimate the cost of building and operaing a UFT, though, here. We describe our cost estimates of a UFT network by cost type and provide an illustrative example in Table \ref{Tab:Cost} for building a UFT network in Chicago using a tunnel budget of 150 miles and one depot. We compute all component costs on a per-package basis in Section \ref{components} and discuss totals and comparisons across different mileage budgets in Section \ref{compare}.

\subsection{Cost components} \label{components}

\subsubsection{Capital expense and other fixed costs}
The capital expense of a UFT network includes the cost of carriages, rail, construction, design, and the motors and systems. For these costs, we use the cost per carriage and distance based cost of rail, construction, design, motors, and systems costs from \cite{magway}. These values are given in Table \ref{Tab:CostCapex}. We determine the number of carriages needed by dividing the number of packages delivered by a UFT network of 150 miles in Chicago by the number of loading windows. The number of packages served by the UFT network in this instance is 373,428. In our computational tests, we assume a maximum loading time of 12 hours, and if we assume four, three hour loading periods, then $\lceil 373,428/4 \rceil$ = 93,357 carriages are needed. These carriages can be re-used in each loading window.
We assume a depreciation period of twelve years based on the expected lifetime of the UFT carriages. We use an interest rate of $8\%$ and an insurance premium of $0.5\%$ to obtain an annual fixed cost of $\$ 804,408 $ per mile of tunnel built. With a total tunnel distance of 300 miles (considering the track going both directions along a given arc) and 373,428 packages delivered by UFT, there is a capital expense and fixed cost of $ \$1.770 $ per package.

\begin{table}[h!] 
\begin{center}
\begin{tabular}{ |l|l|r| } 
 \hline
 {Cost} & {Units} & {Value} \\
 \hline
  Cost/Carriage & $\$$/carriage & 1,172 \\ 
  Rail and Construction & $\$ $/mile & 1,714,438 \\
  Design & $\$ $/mile & 455,685\\
  Motors and System & $\$ $/mile  & 2,243,686\\
 \hline
\end{tabular}
\caption{Capex Cost Data from Magway} \label{Tab:CostCapex}
\end{center}
\end{table}

\subsubsection{Variable costs}
The variable costs associated with a UFT network depend on the energy usage for the carriages and track and the maintenance required for the carriages. In a Magway-like UFT system that is propelled by linear motor technology, the carriages are propelled by components on the track and on the carriages. To run the UFT system, we need to calculate the power that is generated both by the track and by carriages to move them through the system. The power used by the track is approximately 27.35 kW/mile \citep{magway}. To determine the power used by the track over the whole network, we multiply by the total length of tunnel used, which is 300 miles in our example (since there is a tunnel in each direction along each arc included in the network), giving a value for power of 8,204 kW. In our example, the power required for carriages in each time period is $0.8 \text{ kW/carriage} \times 93,357 $ carriages = 74,686 kW. To determine the total energy used by the UFT system in a day, we sum the power of the track and the carriages and multiply by the time over which the system is used. To determine the total time the system is used for each loading window, we add the loading time to twice the time needed for a carriage to reach the microhub located farthest from the depot. We multiply by two to account for the return trip to the depot. In our test case, the travel time to the microhub located farthest from the depot is 1.27 hours. The loading time required in each window is 1.30 hours. This gives a total operational time of $1.30 + 1.27\times2 = 3.846$ hours in each loading window, and over all four loading windows, an operation time of 15.382 hours. Thus, the total energy used by the system in a day will be $(8,204 + 74,686) \times 15.382 = 1,275,030$ kWh. From \cite{energysage}, the cost of energy is $\$0.14$ per kWh, resulting in a cost of $\$178,504$. This gives an energy cost of $\$0.478$ per package.

It is expected \citep{magway} that the wheels on UFT carriages will need to be replaced every 62,137 miles. The cost of each replacement is $ \$25.40 $, which is equivalent to a cost of $ \$0.000409 $ per mile. Using our test case, this equates to a cost of $\$0.123 $ per package. The total variable cost is $ \$0.601 $ per package.

\subsubsection{Staff/labor}
People are needed to operate the depot and microhubs, complete maintenance activities, and engineer software for the UFT system. \cite{magway} estimates that the number of maintenance actions per person per hour is 6, and converting this to a yearly number per person assuming eight hour shifts and 20 work days per month gives 11,520 maintenance actions per person per year. We determine the number of maintenance actions needed per carriage per year using:
\begin{align}
    \frac{(\text{Length of UFT network})\times (365)}{(\text{62,137})\times (\text{Number of carriages})}
\end{align}
where 62,137 is the mileage at which routine carriage maintenance will need to be completed. This gives 0.000179 maintenance actions needed per carriage per year based on our example instance in Chicago. Multiplying this value by the actions per person per year, we find that approximately 42 maintenance actions will be needed per year, and that one maintenance operator is needed.

Based on the required lead time of 0.05 seconds between packages and the length of the carriages, at most eleven packages per second can enter the UFT system. The capacity of a loading/unloading unit is estimated to be 15 s-unit/carriage \citep{magway}. Thus, we need 165 loading/unloading units at the depot so that the maximum loading rate can be achieved. It is also estimated by \cite{magway} that one person can operate five loading/unloading units. We assume that there will be a person at each microhub served by the UFT to unload packages. In our Chicago test case, 114 microhubs are included in the solution, so the total number of staff needed for loading and unloading is $(165/5 + 114) \times 3 = 441$ people. We multiply by three since we expect there to be three eight-hour shifts per day. Additionally, the use of a UFT network requires software engineers. The number needed is estimated to be 30 people \citep{magway}. So, the total number of people needed is 472.

Using this estimate of the number of people needed for warehouse operations and software engineers, we can calculate the staff cost for operations. According to Zip Recruiter, the average hourly wage for warehouse operators is $\$22$ per hour \citep{OperatorWage}. We assume that each person will work an 8 hour shift 20 days a month, giving a total of 1920 hours per person. Multiplying this by the 442 maintenance and operations people needed per year, we have a total of 848,640 hours per year and a yearly cost of $ \$ 18,670,080 $ for warehouse and maintenance staff. The average hourly wage for software engineers in the United States is $\$67$ per hour from \cite{EngineerWage}. For the 30 software engineers needed, a total of 57,600 hours will be worked per year at a cost of $\$22,529,280$. These yearly staff costs for operations can be divided among the deliveries to get a cost of $\$0.165$ per package.

\subsubsection{Loading/Unloading}
While we included the staff costs for loading and unloading above, we have not included the other expenses associated with loading and unloading carriages.   Based on an estimate from \cite{magway}, the operational cost for loading and unloading is $\$127,000$ per loading and unloading unit per year. Above we estimated the need for 165 loading and unloading units, giving a total cost of $\$20,955,000$ per year. Additionally, there will be a cost associated with warehouse leases and operating expenses of approximately $\$11.96$ per square foot according to \cite{WarehouseCost}. We estimate the warehouse surface to be 322,909 square feet \citep{magway}. From this we find the loading/unloading cost per package to be $\$0.182$.

There are a variety of ways that a microhub could be set up, and packages could be distributed to customers. Some suggest that the microhub be set up within existing infrastructure such as transit stations \citep{lee2020moving} or on the side of streets \citep{micro}. As discussed in Section \ref{intro}, many formats and locations for microhubs have been considered without a consensus on the type and size of microhub that should be used.
Most assume a size of 500-1000 square feet. Since we have 114 microhubs, we will assume 114 structures of an average size 750 square feet.  Because they will need less infrastucture than a warehouse and there is NO information out there on pricing for microhubs, we will assume the same cost per square feet as warehouses of \$11.96 per year.  This translates to \$511,290 per year or \$0.0075 per package.  

\subsubsection{Last meter}
 At microhubs, customers may walk up to a counter to retrieve their package, or some may request cargo bike delivery from the microhub. As noted above, we assume there is at least one staff person at each microhub available at all times for unloading packages and/or serving customers. 

We estimate the cost of the last mile for a package based on the use of cargo bikes.  For deliveries that do not use cargo bikes, these costs could be used as proxy for other options. The cost per package of biking 1.5 miles (3/4 mile in each direction) to deliver a package from the microhub is $\$2.828$ from \cite{pedal}. The typical fare includes a fixed pickup fee which we ignored here since the cargo bike operator would likely be picking up multiple packages from a microhub for delivery.

\begin{table}[h!] 
\begin{center}
\begin{tabular}{ |c|l|l|r| } 
 \hline
 {Cost Type} & {Description} & {Units} & {Value} \\
 \hline
  & CAPEX & $\$$ (12 yr depreciation) & 1,433,314,509 \\ 
  & Route Length (both directions) & Miles & 300 \\
  & Number of Carriages & $\lceil \frac{\text{Demand Served}}{\text{\# of loading windows}}\rceil$ & 93,357 \\
  & CAPEX per mile & $\$$/mile & 4,778,661 \\
  & Interest Rate & $\%$ & $8.00\%$  \\
  & Interest & $\$$/mile/year & 382,293 \\
  Fixed Cost & Depreciation Period & years & 12  \\
  & Depreciation & $\$$/mile/year & 398,222 \\
  & Insurance Premium & $\%$ & 0.50$\%$ \\
  & Insurance & $\$$/mile/year & 23,893 \\
  & Fixed Cost per Package & $\$$/package & 1.770 \\
  \hline
  & Power & KWh & 1,275,030 \\
  & Cost of Energy & $\$$/kWh & 0.14 \\
  Variable & Cost of Energy per Package & $\$$/package & 0.478 \\
  Costs & Carriage Maintenance/62,137 miles & $\$$/carriage & 25.40 \\
  & Carriage Maintenance & $\$$/package & 0.123 \\
  \hline
  & Maintenance Actions/operator/hour & actions/person/hour & 6 \\
  & Maintenance Actions/operator/year & actions/person/year & 11520 \\
  & Maintenance Actions/year/carriage & action/year/carriage & 0.000179 \\
  & Total Maintenance Actions per year & actions/year & 67 \\
  & Maintenance Operators & people & 1 \\
  & Carriages loaded/second & carriages/s & 11 \\
  People & Loading capacity of 1 loading unit & s-unit/carriage & 15 \\
  & Loading/Unloading Units & units & 165 \\
  & Units Operated/Operator & units/person & 5 \\
  & Loading/Unloading Operators & people & 441 \\
  & Software Engineers & people & 30 \\
  & Total People Needed & people & 472 \\
  \hline
  & Cost/hour/operator & $\$$/hour & 22 \\
  & Hours worked/year/operator & hr/person & 1,920 \\
  & Total Operator Hours worked/year & hours & 848,640 \\
  Staff Cost& Cost/hour/software engineer & $\$$/hr & 67 \\
  & Total Software Engineer Hours/year & hours & 57,600 \\
  & Cost of people per package & $\$$/package & 0.165 \\
  \hline
  & Yearly cost loading/unloading units & $\$$ & 20,955,000 \\
  Loading/& Cost of Warehouse & $\$$/$\text{ft}^2$/year & 11.96 \\
  Unloading & Warehouse Surface & $\text{ft}^2$ & 322,909 \\
 & Warehouse Cost per package & $\$$/package & 0.182 \\
  &Average Microhub Surface & $\text{ft}^2$ & 750 \\
  &Microhub Cost per package&  \$ & .0075\\
  \hline
  Total UFT Cost & Cost per Package of UFT Delivery & $\$$/package &$2.726$ \\
  \hline
  Last Meter & Cost of Cargo bike & $\$$/package & 2.828 \\ 
 \hline
 Total Delivery Cost & Cost per package & $\$$/package & 5.553 \\
 \hline
\end{tabular}
\caption{UFT Costs for a 1 Depot and Tunnel Budget of 150 miles in Chicago} \label{Tab:Cost}
\end{center}
\end{table}

\subsection{Per package final cost and comparison}\label{compare}
The sample instance with a mileage budget of 150 miles in Table \ref{Tab:Cost} yields a cost per package of \$5.55. This is a very interesting number when we compare it with the traditional cost of last mile delivery. Considering similar costs such as labor, fuel, and maintenance, the cost of last mile delivery for a small package in an urban environment is often estimated as \$10.00 \citep{dispatch}. Labor costs account for the largest expense, and delivery drivers continue to be in short supply \citep{extra}. Thus, using a UFT represents over a 40\% savings in costs, which is an important factor in getting companies to take advantage of all of the other advantages that a UFT brings.  The savings come primarily from reduced labor costs since drivers are not needed to operate a UFT.

To better understand the distance budget's role in the total cost, we examine how the cost per package changes with different mileage budgets and one depot in Figure \ref{fig:costs}. For both Chicago and New York, at all tunnel budgets, the cost per package of delivery by UFT is less than the per package cost of truck delivery in urban environments. At budgets less than 30 miles of tunnels, the cost of delivery per package is larger in Chicago than in New York City due to the smaller volume of packages delivered for the same fixed costs.  This also reflects the distance between the potential depot locations and high demand microhubs in New York City (closer) vs. Chicago (farther).   

We see a minimum cost per package of delivery for New York City at a tunnel budget of 15 miles and for Chicago at 45 miles. With this tunnel length, a relatively high volume of demand per length of tunnel built can be served by the UFT system. This distributes the large fixed cost of tunnel construction over a large number of deliveries while requiring a smaller number of carriages. As the tunnel budget increases, the number of packages per mile of tunnel built decreases, resulting in an increasing cost of UFT delivery per package. This is why we see the continual increase in the cost per package in Chicago, with a cost per package of $\$6.78$ at a tunnel budget of 270 miles. For New York City, we see that the cost per package levels out and has a max value of $\$8.78$ at a tunnel budget of 225 miles. This is due to the loading capacity at the depot becoming binding at tunnel budgets over 120 miles. For these tunnel budgets, small changes in solutions result in small variations in the cost per package. This graph highlights the importance of a city evaluating the solutions of the $n_d$-UFT at different mileage budgets to understand the impact on costs.

\begin{figure}
    \centering
    \includegraphics[width = 0.6\textwidth]{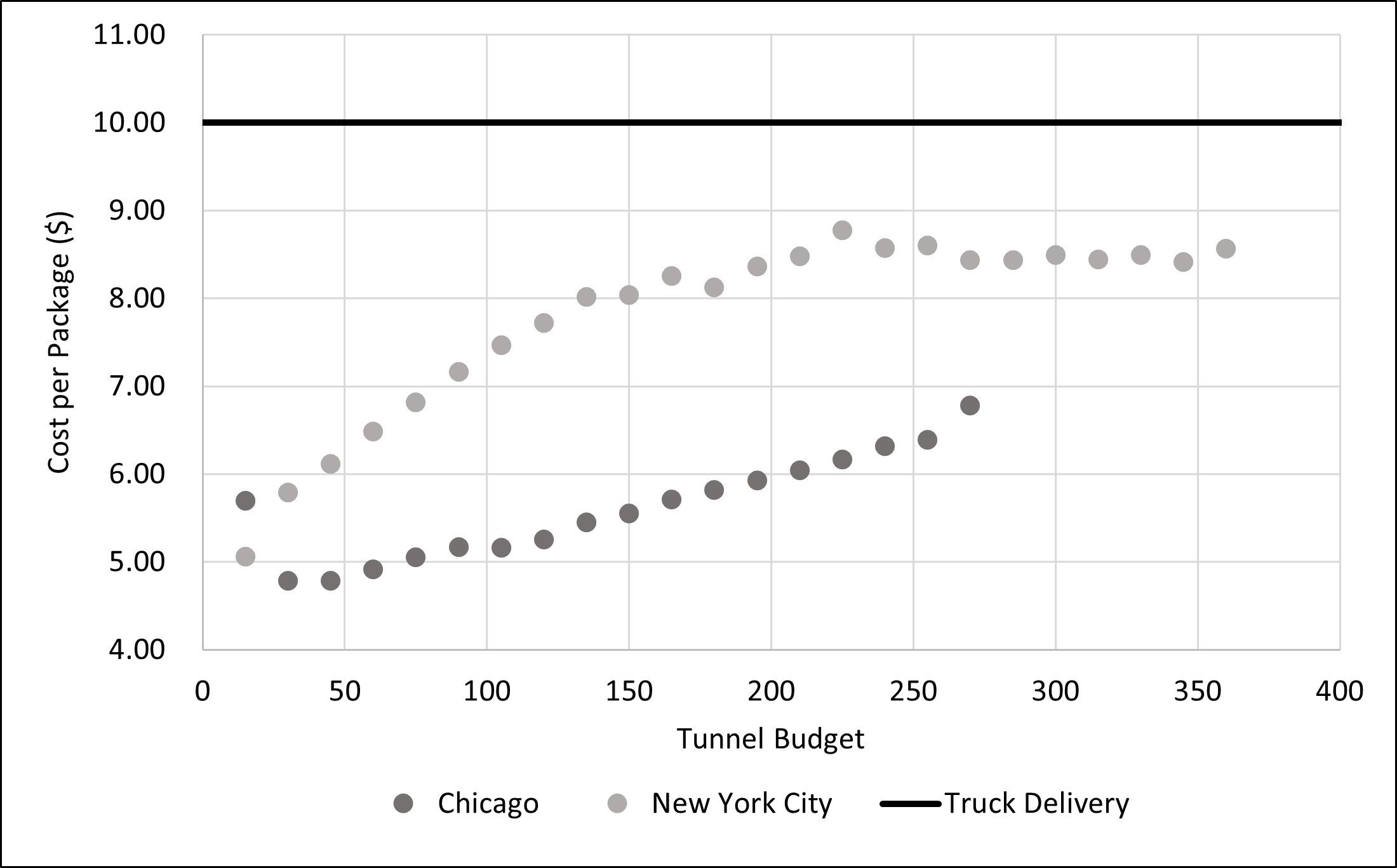}
    \caption{Cost per Package for 1 Depot with a Demand Level of $15\%$}
    \label{fig:costs}
\end{figure}

\section{Conclusions and future work}\label{future}
The $n_d$-UFT provides a promising alternative for package delivery in urban areas that could greatly reduce the number of packages handled  by delivery trucks as well as reduce other important metrics such as emissions. We have also shown that these benefits come from a network where the cost per package of delivery is less than existing delivery systems. Most importantly, we have shown the importance of carefully studying the $n_d$-UFT at different mileage budget levels, different loading capacities, and numbers of depots to understand the important tradeoffs of different network design parameters.  

Many opportunities exist for future work related to the $n_d$-UFT.   There are many unknowns associated with the pricing and vendor acceptance of UFT systems that should be investigated. This includes how vendors should be charged to use a UFT system and the willingness of vendors to adopt such a system.  There is also much to be determined about the best design of microhubs. 

There are many extensions to this problem that could impact the design of the network. Future models could extend the $n_d$-UFT to include capacity at microhubs and to consider the feasibility and costs associated with building a UFT network along specific segments of the road network. Other objectives may be evaluated such as a maximization of truck miles saved or minimization of energy usage. Stochastic settings should also be evaluated, where the demand at microhubs is modeled by a distribution. It is also likely that, in practice, cities would want to build the UFT network in phases or expand the network over time. This extension to the $n_d$-UFT problem could provide cities with UFT network design plans that provide the best use of the system over a multiple year time span.

If UFT networks are built, there are many operational questions that merit careful study.  For example, delivery service providers who use the UFT have to decide when to purchase capacity and how much capacity at each depot.  They also have to coordinate the logistics of separating packages that travel on the UFT vs. trucks and how to handle packages that exceed the size limitations of the carriages. It would also be interesting to utilize the return trips of carriages to the depot to make customer returns or complete the first-mile delivery of packages bound for other cities. In conclusion, UFTs represent a promising new delivery model with many interesting directions for future research.

\bibliographystyle{unsrtnat}
\bibliography{template}  

\begin{thebibliography}{59}
\providecommand{\natexlab}[1]{#1}
\providecommand{\url}[1]{\texttt{#1}}
\expandafter\ifx\csname urlstyle\endcsname\relax
  \providecommand{\doi}[1]{doi: #1}\else
  \providecommand{\doi}{doi: \begingroup \urlstyle{rm}\Url}\fi

\bibitem[{HyperloopTT}(2023)]{HyperloopTT}
{HyperloopTT}.
\newblock Hyperlooptt.
\newblock \url{https://www.hyperlooptt.com/}, 2023.
\newblock Accessed: 2023-09-10.

\bibitem[{Cargo Sous Terrain}(2023)]{cargo}
{Cargo Sous Terrain}.
\newblock \url{https://www.cst.ch/}, 2023.
\newblock Accessed: 2023-04-01.

\bibitem[{Smart City Loop}(2023)]{Smart}
{Smart City Loop}, 2023.
\newblock URL \url{https://www.smartcityloop.de/}.
\newblock Accessed: 2023-04-01.

\bibitem[Boysen et~al.(2023)Boysen, Briskorn, and Rupp]{boysen2023optimization}
Nils Boysen, Dirk Briskorn, and Johannes Rupp.
\newblock Optimization of two-echelon last-mile delivery via cargo tunnel and a delivery person.
\newblock \emph{Computers \& Operations Research}, 151:\penalty0 106123, 2023.

\bibitem[Boysen et~al.(2024)Boysen, Briskorn, Rupp, and Schwerdfeger]{boysen2024jam}
Nils Boysen, Dirk Briskorn, Johannes Rupp, and Stefan Schwerdfeger.
\newblock Jam in the tunnel: On urban freight tunnels, their operational scheduling, and unused transport capacity.
\newblock \emph{Service Science}, 2024.

\bibitem[{Magway}(2023)]{magway}
{Magway}.
\newblock Good(s) delivery: Magway: England.
\newblock https://www.magway.com/, 2023.
\newblock Accessed: 2022-12-08.

\bibitem[{New York City Government}(2023)]{NYC}
{New York City Government}.
\newblock Microhubs pilot.
\newblock \url{https://www.nyc.gov/html/dot/downloads/pdf/microhubs-pilot-report.pdf}, 2023.
\newblock Accessed: 2023-12-08.

\bibitem[Beling(2023)]{micro}
Sarah Beling.
\newblock Microhub delivery program pilot could send parcel parking problems packing.
\newblock https://w42st.com/post/microhub-delivery-program-pilot-could-send-parcel-parking-problems-packing/, 2023.
\newblock Accessed: 2023-09-11.

\bibitem[Millard(2022)]{south}
Scott Millard.
\newblock South west freight strategy.
\newblock Technical report, July 2022.

\bibitem[Locus(2021)]{Locus}
Locus.
\newblock Build an efficient urban logistics with microhubs.
\newblock \url{https://blog.locus.sh/microhubs-for-sustainable-urban-freight-logistics/}, 2021.
\newblock Accessed: 2023-12-21.

\bibitem[Lee et~al.(2019)Lee, Kim, and Wiginton]{lee}
Janelle Lee, Carolyn Kim, and Lindsay Wiginton.
\newblock Delivering last-mile solutions: A feasibility analysis of microhubs and cyclelogistics in the gtha, 2019.
\newblock URL \url{http://www.jstor.org/stable/resrep22106.6}.

\bibitem[Gansterer and Hartl(2020)]{gansterer2020shared}
Margaretha Gansterer and Richard~F Hartl.
\newblock Shared resources in collaborative vehicle routing.
\newblock \emph{Top}, 28:\penalty0 1--20, 2020.

\bibitem[Fischetti et~al.(1994)Fischetti, Hamacher, J{\o}rnsten, and Maffioli]{fischetti1994weighted}
Matteo Fischetti, Horst~W Hamacher, Kurt J{\o}rnsten, and Francesco Maffioli.
\newblock Weighted k-cardinality trees: Complexity and polyhedral structure.
\newblock \emph{Networks}, 24\penalty0 (1):\penalty0 11--21, 1994.

\bibitem[Alumur et~al.(2021)Alumur, Campbell, Contreras, Kara, Marianov, and O’Kelly]{alumur2021perspectives}
Sibel~A Alumur, James~F Campbell, Ivan Contreras, Bahar~Y Kara, Vladimir Marianov, and Morton~E O’Kelly.
\newblock Perspectives on modeling hub location problems.
\newblock \emph{European Journal of Operational Research}, 291\penalty0 (1):\penalty0 1--17, 2021.

\bibitem[Visser(2018)]{visser2018development}
Johan~GSN Visser.
\newblock The development of underground freight transport: An overview.
\newblock \emph{Tunnelling and Underground Space Technology}, 80:\penalty0 123--127, 2018.

\bibitem[Chen et~al.(2018)Chen, Guo, Chen, Fan, and Li]{chen2018using}
Yicun Chen, Dongjun Guo, Zhilong Chen, Yiqun Fan, and Xiao Li.
\newblock Using a multi-objective programming model to validate feasibility of an underground freight transportation system for the yangshan port in shanghai.
\newblock \emph{Tunnelling and Underground Space Technology}, 81:\penalty0 463--471, 2018.

\bibitem[Ehsan~Zahed et~al.(2017)Ehsan~Zahed, Mohsen~Shahandasht, and Najafi]{ehsan2017investment}
S~Ehsan~Zahed, S~Mohsen~Shahandasht, and Mohammad Najafi.
\newblock Investment valuation of an underground freight transportation (uft) system in texas.
\newblock In \emph{Pipelines 2017}, pages 192--201. 2017.

\bibitem[Liu(2004)]{liu2004feasibility}
Henry Liu.
\newblock Feasibility of using pneumatic capsule pipelines in new york city for underground freight transport.
\newblock In \emph{Pipeline Engineering and Construction: What's on the Horizon?}, pages 1--12. 2004.

\bibitem[Dong et~al.(2019)Dong, Xu, Hwang, Ren, and Chen]{dong2019impact}
Jianjun Dong, Yuanxian Xu, Bon-gang Hwang, Rui Ren, and Zhilong Chen.
\newblock The impact of underground logistics system on urban sustainable development: A system dynamics approach.
\newblock \emph{Sustainability}, 11\penalty0 (5):\penalty0 1223, 2019.

\bibitem[{Hyperloop One}(2016)]{Hyperloop}
{Hyperloop One}.
\newblock Hyperloop one.
\newblock \url{https://www.prnewswire.com/news-releases/hyperloop-one-fs-links-and-kpmg-publish-worlds-first-study-of-full-scale-hyperloop-system-300294040.html}, 2016.
\newblock Accessed: 2023-04-01.

\bibitem[Merchant and Chankov(2020)]{merchant2020towards}
Deep~Vijay Merchant and Stanislav~M Chankov.
\newblock Towards a european hyperloop network: An alternative to air and rail passenger travel.
\newblock In \emph{2020 IEEE International Conference on Industrial Engineering and Engineering Management (IEEM)}, pages 128--132. IEEE, 2020.

\bibitem[Shah(2019)]{shah2019hyperloop}
Kalrav~Mehulkumar Shah.
\newblock Hyperloop network design: The swiss case.
\newblock Master's thesis, Delft University of Technology, Delft, Netherlands, November 2019.

\bibitem[Gansterer and Hartl(2018)]{gansterer2018collaborative}
Margaretha Gansterer and Richard~F Hartl.
\newblock Collaborative vehicle routing: A survey.
\newblock \emph{European Journal of Operational Research}, 268\penalty0 (1):\penalty0 1--12, 2018.

\bibitem[Zhang et~al.(2020)Zhang, Chen, Zhang, Wang, Yang, and Cai]{zhang2020composite}
Wenyu Zhang, Zixuan Chen, Shuai Zhang, Weirui Wang, Shuiqing Yang, and Yishuai Cai.
\newblock Composite multi-objective optimization on a new collaborative vehicle routing problem with shared carriers and depots.
\newblock \emph{Journal of Cleaner Production}, 274:\penalty0 122593, 2020.

\bibitem[Ko et~al.(2020)Ko, Sari, Makhmudov, and Ko]{ko2020collaboration}
Seung~Yoon Ko, Ratna~Permata Sari, Muzaffar Makhmudov, and Chang~Seong Ko.
\newblock Collaboration model for service clustering in last-mile delivery.
\newblock \emph{Sustainability}, 12\penalty0 (14):\penalty0 5844, 2020.

\bibitem[Handoko et~al.(2014)Handoko, Nguyen, and Lau]{handoko2014auction}
Stephanus~Daniel Handoko, Duc~Thien Nguyen, and Hoong~Chuin Lau.
\newblock An auction mechanism for the last-mile deliveries via urban consolidation centre.
\newblock In \emph{2014 IEEE International Conference on Automation Science and Engineering (CASE)}, pages 607--612. IEEE, 2014.

\bibitem[Even et~al.(1975)Even, Itai, and Shamir]{even1975complexity}
Shimon Even, Alon Itai, and Adi Shamir.
\newblock On the complexity of time table and multi-commodity flow problems.
\newblock In \emph{16th annual symposium on foundations of computer science (sfcs 1975)}, pages 184--193. IEEE, 1975.

\bibitem[Ford~Jr and Fulkerson(1958)]{ford1958suggested}
Lester~Randolph Ford~Jr and Delbert~R Fulkerson.
\newblock A suggested computation for maximal multi-commodity network flows.
\newblock \emph{Management Science}, 5\penalty0 (1):\penalty0 97--101, 1958.

\bibitem[Hu(1963)]{hu1963multi}
T~Chiang Hu.
\newblock Multi-commodity network flows.
\newblock \emph{Operations research}, 11\penalty0 (3):\penalty0 344--360, 1963.

\bibitem[Salimifard and Bigharaz(2022)]{salimifard2022multicommodity}
Khodakaram Salimifard and Sara Bigharaz.
\newblock The multicommodity network flow problem: state of the art classification, applications, and solution methods.
\newblock \emph{Operational Research}, pages 1--47, 2022.

\bibitem[Ouorou et~al.(2000)Ouorou, Mahey, and Vial]{ouorou2000survey}
Adamou Ouorou, Philippe Mahey, and J-Ph Vial.
\newblock A survey of algorithms for convex multicommodity flow problems.
\newblock \emph{Management Science}, 46\penalty0 (1):\penalty0 126--147, 2000.

\bibitem[Assad(1978)]{assad1978multicommodity}
Arjang~A Assad.
\newblock Multicommodity network flows—a survey.
\newblock \emph{Networks}, 8\penalty0 (1):\penalty0 37--91, 1978.

\bibitem[Yaghini and Akhavan(2012)]{yaghini2012multicommodity}
Masoud Yaghini and Rahim Akhavan.
\newblock Multicommodity network design problem in rail freight transportation planning.
\newblock \emph{Procedia-Social and Behavioral Sciences}, 43:\penalty0 728--739, 2012.

\bibitem[Rudi et~al.(2016)Rudi, Fr{\"o}hling, Zimmer, and Schultmann]{rudi2016freight}
Andreas Rudi, Magnus Fr{\"o}hling, Konrad Zimmer, and Frank Schultmann.
\newblock Freight transportation planning considering carbon emissions and in-transit holding costs: a capacitated multi-commodity network flow model.
\newblock \emph{EURO Journal on Transportation and Logistics}, 5\penalty0 (2):\penalty0 123--160, 2016.

\bibitem[Hellsten et~al.(2021)Hellsten, Koza, Contreras, Cordeau, and Pisinger]{hellsten2021transit}
Erik Hellsten, David~Franz Koza, Ivan Contreras, Jean-Fran{\c{c}}ois Cordeau, and David Pisinger.
\newblock The transit time constrained fixed charge multi-commodity network design problem.
\newblock \emph{Computers \& Operations Research}, 136:\penalty0 105511, 2021.

\bibitem[Crainic(2000)]{crainic2000service}
Teodor~Gabriel Crainic.
\newblock Service network design in freight transportation.
\newblock \emph{European journal of operational research}, 122\penalty0 (2):\penalty0 272--288, 2000.

\bibitem[Wang(2018)]{wang2018multicommodity}
I-Lin Wang.
\newblock Multicommodity network flows: A survey, part i: Applications and formulations.
\newblock \emph{International Journal of Operations Research}, 15\penalty0 (4):\penalty0 145--153, 2018.

\bibitem[Cook et~al.(2011)Cook, Applegate, Bixby, and Chvatal]{cook2011traveling}
William~J Cook, David~L Applegate, Robert~E Bixby, and Vasek Chvatal.
\newblock \emph{The traveling salesman problem: a computational study}.
\newblock Princeton University Press, 2011.

\bibitem[{Transport Research Labratory}(2023)]{TRL}
{Transport Research Labratory}.
\newblock Transport reasearch labratory | the future of transport.
\newblock https://www.trl.co.uk/, 2023.
\newblock Accessed: 2023-08-31.

\bibitem[Kandyba-Chimani(2011)]{kandyba2011exact}
Maria Kandyba-Chimani.
\newblock \emph{Exact algorithms for Network Design Problems using Graph Orientations}.
\newblock PhD thesis, Technische Universität, 2011.

\bibitem[Andres~Duany(2021)]{15min}
Rovert~Steuteville Andres~Duany.
\newblock Defining the 15-minute city.
\newblock \url{https://www.cnu.org/publicsquare/2021/02/08/defining-15-minute-city}, 2021.
\newblock Accessed: 2023-02-01.

\bibitem[{OpenStreetMap contributors}(2017)]{OpenStreetMap}
{OpenStreetMap contributors}.
\newblock {Planet dump retrieved from https://planet.osm.org }.
\newblock \url{ https://www.openstreetmap.org }, 2017.

\bibitem[Peng and Murray(2020)]{veroviz2020}
Lan Peng and Chase Murray.
\newblock {VeRoViz}: A vehicle routing visualization toolkit.
\newblock \url{https://ssrn.com/abstract=3746037}, 2020.
\newblock Accessed: 2023-03-07.

\bibitem[{UK Department of Transportation}(2023)]{UKdft}
{UK Department of Transportation}.
\newblock Transport statistics finder, goods lifted and goods moved by commodity and type and weight of vehicle.
\newblock \url{https://maps.dft.gov.uk/transport-statistics-finder/index.html}, 2023.
\newblock Accessed: 2023-09-13.

\bibitem[{City of Chicago}(2023)]{ZipCodeBoundaries}
{City of Chicago}.
\newblock Chicago data portal.
\newblock \url{https://data.cityofchicago.org/Facilities-Geographic-Boundaries/Boundaries-ZIP-Codes/gdcf-axmw}, 2023.
\newblock Accessed: 2023-02-01.

\bibitem[{Chicago Data Portal}(2018)]{PopulationData}
{Chicago Data Portal}.
\newblock Chicago data portal: Chicago population counts.
\newblock \url{https://data.cityofchicago.org/Health-Human-Services/Chicago-Population-Counts/85cm-7uqa/data}, 2018.
\newblock Accessed: 2023-02-01.

\bibitem[{City of New York}(2022)]{NYCPopulationData}
{City of New York}.
\newblock Demographic statistics by zip code.
\newblock \url{https://data.cityofnewyork.us/City-Government/Demographic-Statistics-By-Zip-Code/kku6-nxdu}, 2022.
\newblock Accessed: 2023-02-01.

\bibitem[Haag and Hu(2019)]{haag20191}
Matthew Haag and Winnie Hu.
\newblock 1.5 million packages a day: The internet brings chaos to ny streets.
\newblock \emph{The New York Times}, 28, 2019.

\bibitem[{Environmental Defense Fund}(2015)]{EnvironmentalDefenseFund}
{Environmental Defense Fund}.
\newblock Edf green freight handbook.
\newblock \url{https://business.edf.org/insights/green-freight-math-how-to-calculate-emissions-for-a-truck-move/}, 2015.
\newblock Accessed: 2023-09-18.

\bibitem[{Transport and Environment}(2022)]{TransportAndEnvironment}
{Transport and Environment}.
\newblock Road freight: Vans.
\newblock \url{https://www.transportenvironment.org/challenges/road-freight/vans/}, 2022.
\newblock Accessed: 2023-09-18.

\bibitem[Amazon(2023)]{amazon}
Amazon.
\newblock A delivery driver for amazon gives us a tour of his van—see what it's like inside.
\newblock \url{https://www.aboutamazon.com/news/transportation/see-inside-amazon-dsp-delivery-van}, 2023.
\newblock Accessed: 2023-09-18.

\bibitem[{Energy Sage}(2023)]{energysage}
{Energy Sage}.
\newblock Cost of electricity in iowa.
\newblock https://www.energysage.com/, 2023.
\newblock Accessed: 2023-09-10.

\bibitem[{Zip Recruiter}(2023{\natexlab{a}})]{OperatorWage}
{Zip Recruiter}.
\newblock Warehouse maintenance technician salary.
\newblock \url{https://www.ziprecruiter.com/Salaries/Warehouse-Maintenance-Technician-Salary}, 2023{\natexlab{a}}.
\newblock Accessed: 2023-09-10.

\bibitem[{Zip Recruiter}(2023{\natexlab{b}})]{EngineerWage}
{Zip Recruiter}.
\newblock How much do software engineer jobs pay per hour?
\newblock \url{https://www.ziprecruiter.com/Salaries/Software-Engineer-Salary-per-Hour}, 2023{\natexlab{b}}.
\newblock Accessed: 2023-09-10.

\bibitem[{Warehousing and Fullfillment}(2023)]{WarehouseCost}
{Warehousing and Fullfillment}.
\newblock Warehousing services costs, pricing, rates and fees.
\newblock \url{https://www.warehousingandfulfillment.com/resources/warehousing-services-costs-pricing-rates-and-fees/}, 2023.
\newblock Accessed: 2023-09-10.

\bibitem[Lee and Morales(2020)]{lee2020moving}
Janelle Lee and Vincent Morales.
\newblock Moving edmonton to efficient, low-carbon, urban freight delivery.
\newblock https://www.pembina.org/pub/moving-edmonton-efficient-low-carbon-urban-freight-delivery/, 2020.
\newblock Accessed: 2023-09-19.

\bibitem[{PedalMe}(2023)]{pedal}
{PedalMe}.
\newblock Frequently asked questions.
\newblock https://pedalme.co.uk/support, 2023.
\newblock Accessed: 2023-09-11.

\bibitem[{Dispatch Track}(2021)]{dispatch}
{Dispatch Track}.
\newblock Breaking down last mile delivery costs.
\newblock https://www.dispatchtrack.com/blog/last-mile-delivery-costs-breakdown, 2021.
\newblock Accessed: 2023-09-11.

\bibitem[Ward(2022)]{extra}
C.~Ward.
\newblock Delivery driver shortage: Top reasons, impacts \& best solutions.
\newblock 2022.
\newblock Accessed: 2023-09-11.

\end{thebibliography}

\end{document}